\documentclass[12pt,reqno]{amsart}

\usepackage{amssymb,latexsym}
\usepackage{amsmath}
\usepackage{amsthm}
\usepackage{rotating}
\usepackage{graphicx}
\usepackage{latexsym}
\usepackage{amsfonts}
\usepackage{graphics}
\usepackage{color}
\usepackage{eufrak}
\usepackage{amstext}
\usepackage{amsopn}
\usepackage{amsbsy}
\usepackage{amscd}
\usepackage{amsxtra}
\usepackage{enumerate}
\usepackage{titletoc}
\usepackage{amsfonts}

\newcommand{\overbar}[1]{\mkern 1.5mu\overline{\mkern-1.5mu#1\mkern-1.5mu}\mkern 1.5mu}

\newtheorem{theorem}{Theorem}[section]
\newtheorem{proposition}[theorem]{Proposition}
\newtheorem{lemma}[theorem]{Lemma}
\newtheorem{corollary}[theorem]{Corollary}

\newtheorem*{conjecture}{Conjecture \cite{WY-10}}

\newtheorem{openproblem}{Question}

\newtheorem{definition}[theorem]{Definition}

\newtheorem{remark}{Remark}

\theoremstyle{definition}

\numberwithin{equation}{section}

\theoremstyle{remark}

\title[Infinitely many positive solutions]{Infinitely many positive solutions of nonlinear Schr\"{o}dinger equations with  non-symmetric potentials}

\author{Manuel del Pino}
\address{\noindent M. del Pino - Departamento de Ingenier\'{\i}a  Matem\'atica and Centro de
Modelamiento Matem\'atico (UMI 2807 CNRS), Universidad de Chile, Casilla 170 Correo 3,
Santiago, Chile}
\email{delpino@dim.uchile.cl}

\author{Juncheng Wei}
\address{\noindent J. Wei - Department of Mathematics, University of British Columbia, Vancouver, BC V6T 1Z2, Canada and Department of Mathematics, Chinese University of Hong Kong, Shatin, NT, Hong Kong}
\email{jcwei@math.ubc.ca}

\author{Wei Yao}
\address{\noindent W.Yao - Departamento de Ingenier\'{\i}a  Matem\'atica and Centro de
Modelamiento Matem\'atico (UMI 2807 CNRS), Universidad de Chile, Casilla 170 Correo 3,
Santiago, Chile}
\email{wyao.cn@gmail.com}

\date{}

\begin{document}

\maketitle

\begin{abstract}

We consider the standing-wave problem for a nonlinear Schr\"{o}dinger equation, corresponding to the semilinear elliptic problem
\begin{equation*}
-\Delta u+V(x)u=|u|^{p-1}u,\ u\in H^1(\mathbb{R}^2),
\end{equation*}
where $V(x)$ is a uniformly positive potential and $p>1$.
Assuming that \begin{equation*}
V(x)=V_\infty+\frac{a}{|x|^m}+O\Big(\frac{1}{|x|^{m+\sigma}}\Big),\ \text{as}\ |x|\rightarrow+\infty,
\end{equation*}
for instance if $p>2$, $m>2$ and $\sigma>1$  we prove the existence of infinitely many positive solutions.
If $V(x)$ is radially symmetric, this result was proved in \cite{WY-10}. The proof without symmetries is much more difficult, and for that
 we develop a new {\em intermediate Lyapunov-Schmidt reduction method}, which is a compromise between the finite and infinite dimensional
 versions of it.


\end{abstract}

\tableofcontents


\section{Introduction and statement of the main result}


In this paper we consider the problem of finding positive solutions of the classical semilinear elliptic problem
\begin{equation}\label{eq-V-u}
-\Delta u+V(x)u=|u|^{p-1}u\ \text{in}\ \mathbb{R}^N,
\end{equation}
where $\Delta=\sum_{j=1}^N\frac{\partial^2}{\partial x_j^2}$ stands for the Laplace operator in  $\mathbb{R}^N$, $V(x)$ is a non-negative potential, and $p>1$.


\medskip

Equation \eqref{eq-V-u} arises in various branches of applied mathematics and physics (cf. \cite{BL-83} and references therein). For instance, in condensed matter physics one simulates the interaction effect among many particles to obtain a focusing nonlinear Schr\"{o}dinger equation of the form
\begin{align}\label{NLS}
i\hbar\frac{\partial{\psi}}{\partial t}=-\hbar^2\Delta{\psi}+ W(x){\psi}-|{\psi}|^{p-1}{\psi}\ \text{in}\ [0,\infty)\times\mathbb{R}^N,
\end{align}
where $i$ is the imaginary unit, $\hbar$ the Planck constant and $W(x)$ a given potential.  
{\em Standing wave solutions} of \eqref{NLS} are those of the form
$${\psi}(t,x)=e^{-i\lambda t/\hbar}u(\hbar^{-1}x) $$  where $u(x)$ is a real-valued function. 
Then \eqref{NLS}  reduces to equation~\eqref{eq-V-u} for $u$, where  $V(x)= W(\hbar x)-\lambda$.


\medskip
In what follows, we shall only consider positive, finite energy solutions of \eqref{eq-V-u}. Namely, we are concerned with the problem:
\begin{align}\label{eq-u}
\left\{
\begin{array}{ll}
-\Delta u+V(x) u-u^p=0\ \text{in}\ \mathbb{R}^N,
\vspace{1mm}\\
u>0\ \text{in}\ \mathbb{R}^N,\ u\in H^1(\mathbb{R}^N).
\end{array}\right.
\end{align}
Associated to \eqref{eq-u} is the energy functional
\begin{equation}\label{energy}
\mathcal{E}(u)=\frac{1}{2}\int_{\mathbb{R}^N}\Big\{|\nabla u|^2+V(x)u^2\Big\}dx-\frac{1}{p+1}\int_{\mathbb{R}^N}u_+^{p+1}\,dx,
\end{equation}
where $u_+=\max\{u,0\}$.
In all what follows, we make the following structure assumptions on $V$ and $p$:
\begin{equation} \hbox{  $V$ is locally H\"older continuous, $V\in L^\infty(\mathbb{R}^N)$ and $V_0=\inf_{x\in\mathbb{R}^N}V(x)>0$} \label{(V0)}\end{equation}
\begin{equation} \hbox{
 $1<p<\infty$ for $N=2$ and $1<p<\frac{N+2}{N-2}$ for $N\geq3$.
} \qquad\qquad\qquad\qquad\quad\label{(P)}\end{equation}

\medskip
Under these hypotheses, it is standard that classical solutions of \eqref{eq-u} correspond precisely to non-trivial critical points of
$\mathcal{E}$ in  $H^1(\mathbb{R}^N)$.

\medskip
Let us denote the set of solutions of problem \eqref{eq-u} by $\mathcal{S}_V$. A natural question is whether or not $\mathcal{S}_V\neq\emptyset$. When $V$ is radially symmetric, the answer is yes (cf. Theorem 4.6 in \cite{DN-86}). But the answer is no when the potential is increasing along a direction (cf. Theorem 1.1 in \cite{CM-10}).

If we further assume that
\begin{equation}\label{V-infinity}
\lim_{|x|\rightarrow\infty}V(x)=V_\infty>0,
\end{equation}
the existence of a positive solution of \eqref{eq-u} has been widely investigated. For example, if we further suppose that
\begin{align}\label{cond-V-1}
\inf_{x\in\mathbb{R}^N}V(x)<V_\infty,
\end{align}
then one can show that \eqref{eq-u} has a least energy (ground state) solution by using the concentration compactness principle (cf. \cite{Lions-84-I,Lions-84-II,DN-86,R-92}). But if \eqref{cond-V-1} does not hold, problem \eqref{eq-u} may not have a least energy solution and solutions have to be seeked for at higher energy levels.  Results in this direction are contained in \cite{BLi-90,BL-97,Cao-90}, where  a positive solution has been found by  variational methods 
under a suitable decay condition on $V$ at infinity.

\medskip
The structure of the solution set $\mathcal{S}_V$ may be quite rich and interesting.  Let us consider for instance the semi-classical limit case:
\begin{align}\label{eq-u-h}
\left\{
\begin{array}{ll}
-\varepsilon^2\Delta u+ W(x) u-u^p=0\ \text{in}\ \mathbb{R}^N,
\vspace{1mm}\\
u>0\ \text{in}\ \mathbb{R}^N,\ u\in H^1(\mathbb{R}^N),
\end{array}\right.
\end{align}
where $\varepsilon>0$ is a small parameter. Naturally, problem~\eqref{eq-u-h} is equivalent to problem~\eqref{eq-u} for $V(x)= W(\varepsilon x)$. It is known that as $\varepsilon$ goes to zero, highly concentrated solutions near critical points of the potential $W$  can be found, see  \cite{ABC-97,CNY-98,CNY-99} \cite{dPF-96}-\cite{dPF-02}, \cite{FW-86,KW-00,Oh-90,Wang-93},  or near higher dimensional stationary sets of other auxiliary potentials \cite{AMN-03,dPKW-07,MMM-09,WWY-11}. The number of solutions of \eqref{eq-u-h} may depend on the number or type of the critical points of $\widetilde{V}(x)$.
It is rather difficult task to understand the structure of $\mathcal{S}_V$ for an arbitrary potential $V$.
For instance, a conspicuously  unanswered question is whether or not  $\mathcal{S}_V\neq\emptyset$ for any potential $V$ satisfying \eqref{V-infinity}.

\medskip
Summing up, the above-mentioned work concern the existence of positive solutions, i.e., $\mathcal{S}_V\neq\emptyset$.  There is less work on the multiplicity of positive solutions of \eqref{eq-u}, namely on estimating $\#(\mathcal{S}_V)$. A seminal  result in this direction was given in  Coti-Zelati and Rabinowitz \cite{CR-92} where $V(x)$ is spatially periodic. In that situation they prove the existence of infinitely many positive solutions, distinct up to periodic translations, via variational methods.

\medskip

Recently, by assuming that $V=V(|x|)$ is radially symmetric, the second author and Yan~\cite{WY-10} proved that problem \eqref{eq-u} has infinitely many positive non-radial solutions if there are constants $V_\infty>0$, $a>0$, $m>1$, and $\sigma>0$, such that
\begin{align}\label{cond-wy-r}
V(r)=V_\infty+\frac{a}{r^m}+O\Big(\frac{1}{r^{m+\sigma}}\Big),\ \text{as}\ r\rightarrow+\infty.
\end{align}

\medskip
An alternative proof through min-max methods was given by Devillanova and Solimini \cite{DS}.

\medskip
The proof in \cite{WY-10} uses in essential way the radial symmetry of the potential $V$. On the other hand, it is conjectured
there that the result should remain true when the symmetry requirement is lifted:

\begin{conjecture}
Problem \eqref{eq-u} has infinitely many positive solutions if
there are constants $V_\infty>0$, $a>0$, $m>1$, and $\sigma>0$, such that the potential $V(x)$ satisfies
\begin{equation}\label{cond-wy}
V(x)=V_\infty+\frac{a}{|x|^m}+O\Big(\frac{1}{|x|^{m+\sigma}}\Big),\ \text{as}\ |x|\rightarrow+\infty.
\end{equation}
\end{conjecture}
\noindent Results in this direction with non-symmetric potentials, as far as we know, there are only perturbative results (cf. \cite{CDS-11,AW-12}). For instance, if $V(x)$ tends to $V_\infty$ from above with a suitable rate:
\begin{equation}
\label{newcond}
V(x) \geq V_\infty>0,\ \lim_{|x|\rightarrow\infty}\big(V(x)-V_\infty\big)e^{\bar{\eta}|x|}=+\infty,\ \text{for some}\ \bar{\eta}\in(0,\sqrt{V_\infty}),
\end{equation}
and $V$ satisfies a global condition:
\begin{align}\label{cond-CDS-0}
\sup_{x\in\mathbb{R}^N}\|V(x)-V_\infty\|_{L^{N/2}(B_1(x))}<\mathcal{V},
\end{align}
where $\mathcal{V}$ is a sufficiently small positive constant (with no explicit expression), Cerami, Passaseo and  Solimini  \cite{CDS-11} proved that problem \eqref{eq-u} has infinitely many positive solutions by purely variational methods. (In \cite{AW-12}, Ao and Wei gave a new proof of this result,   using localized energy method. The new techniques also allow them to deal with more general nonlinearity.)

\medskip

The main purpose of this paper is to prove the above conjecture under some additional assumptions.  In \cite{WY-10}, the  fact that $V$ is radially symmetry allows to build a $k$-bump solution for an arbitrary $k\ge 1$ with a $k$-dyadic symmetry, reducing the problem to just adjusting one parameter representing the location of a single bump along a given ray. A finite dimensional Lyapunov-Schmidt reduction method is used.

\medskip
When $V$ is non-symmetric, we cannot constrain the bump configuration to any symmetry class. We are thus forced to deal with a  large number of bumps and therefore with a huge number of parameters which need to be adjusted.
This poses a tremendous difficulty in the construction comparatively to \cite{WY-10}. In Lyapunov-Schmidt reduction for problems like \eqref{eq-u}
the situation of adjusting a finite number of points ({\em finite dimensional Lyapunov-Schmidt reduction method}), and that of adjusting a higher dimensional object such a geodesic in a suitable metric
as limiting concentration sets ({\em infinite dimensional Lyapunov-Schmidt reduction method}) have been treated.  In this work we develop an {\em intermediate  Lyapunov-Schmidt reduction method}, which consists
 of the finite dimensional procedure for large number of reduced equations, which in the limit become an ODE  system of limiting Jacobi-type operators (see (\ref{cont}) below).  Treating the discrete problem needs a method, technically delicate, which we interpret as an intermediate procedure between the finite and the infinite dimensional one (see for instance \cite{dkw-DG}, \cite{dPKW-07}, \cite{pacard-ritore} and references therein for the latter). The main difference between the intermediate and infinite dimensional reduction, is that in the latter procedure only the variations in the normal direction are needed so the usual Jacobi operator for a curve appears. In the former procedure we also need to take into account  variations in the tangential direction of points, which in the limit may be interpreted as a reparametrization of the curve. This seems to be a {new} procedure, with potentially many interesting applications.

\medskip
 Our main result removes the symmetry assumption on $V$  when $N=2$.

\medskip

\begin{theorem}\label{thm-0}
Let $N=2$. Suppose that $V(x)$ satisfies \eqref{(V0)} and \eqref{cond-wy} for some constants $V_\infty>0$, $a>0$, and
\begin{align}\label{cond-m}
\min\big\{1,\frac{p-1}{2}\big\}m>2,\ \sigma>2.
\end{align}

Then problem \eqref{eq-u} has infinitely many non-radial positive solutions, whose energy can be made arbitrarily large.

\end{theorem}

\medskip

  If \eqref{cond-wy} holds in the $C^1$ sense, then ``$\sigma>2$" in \eqref{cond-m} can be improved to be ``$\sigma>1$". The condition on $p$ can be further relaxed if we assume more regularity of the condition \eqref{cond-wy} or if $p$ is an integer.

The results in \cite{CDS-11,AW-12} make us think that condition \eqref{cond-wy} could be improved. In fact we believe that the optimal condition should be (\ref{newcond}). We stress that our result does not require a global perturbative assumption such as \eqref{cond-CDS-0} on $V$.
In addition, it is worth pointing out that the results on the existence of positive solutions in \cite{BLi-90,BL-97,Cao-90} do not include the polynomial decay case \eqref{cond-wy}.

Finally, we remark that for $N\geq3$, Theorem \ref{thm-0} holds if  we assume the following additional symmetry assumption on $V$: after suitably rotating the coordinate system,
\begin{align}\label{cond-even-1}
V(x)=V(x',x'')=V(x',-x''),
\end{align}
where $x=(x',x'')\in\mathbb{R}^2\times\mathbb{R}^{N-2}$. An open question is whether or not the same result holds when $N \geq 3$  with no extra assumption made.

\medskip

Throughout the paper, we shall use  the following notation and conventions:

\vspace{1mm}
\begin{itemize}
\item For quantities $A_K$ and $B_K$, we write $A_K\sim B_K$ to denote that there exists a positive constant $C$ such that $1/C\leq A_K/B_K\leq C$ for $K$ sufficiently large; $A_K=O(B_K)$ means that $|A_K/B_K|$ are uniformly bounded as $K$ tends to infinity; $A_K=o(B_K)$ denotes that $|A_K/B_K|\rightarrow0$ as $K\rightarrow\infty$.

\item For simplicity, the letter $C$ denotes various generic constant which is independent of $K$. It is allowed to vary from line to line, and also within the same formula.

\item We will use the same $|y|=\|y\|_2$ for the Euclidean norm in various Euclidean spaces $\mathbb{R}^N$ when no confusion can arise and we always denote the inner product of $a$ and $b$ in $\mathbb{R}^N$ by $a\cdot b$.

\item For the index $j\in\{1,2,\dots,K\}$, we shall always use the convention that $j-1=K$ if $j=1$ and $j+1=1$ if $j=K$.

\item The cardinality of a finite set $E$ will be denoted by $\#E$; The Lebesgue measure of a set $E\subset\mathbb{R}^N$ will be denoted by $|E|$.

\item The transpose of a matrix $A$ will be denoted by $A^T$.

\item For each function $w(x)$ defined in $\mathbb{R}^N$, if $w$ is radially symmetric, then there is a real function $\widetilde{w}(r)$ such that $w(x)=\widetilde{w}(|x|)$. With slight abuse of notation, we will simply write $w(r)$ instead of $\widetilde{w}(r)$.
\end{itemize}

In the next section, we will describe the procedure of our construction and give the main ideas of each step.


\medskip

\noindent {\bf Acknowledgments:}
The authors are grateful to Frank Pacard for sharing his ideas and for useful discussions. Manuel del Pino is supported by
 Fondecyt grant 110181 and Fondo Basal CMM.
J. Wei is supported by a NSERC from Canada. The research of W. Yao is supported by Fondecyt Grant 3130543.


\section{Description of the construction}

We will prove the slightly more general version of the Theorem, for $N\ge 3$
where we assume even symmetry in $N-2$ the remaining variables in the sense that
\begin{align}\label{cond-even-11}
V(x)=V(x',x'')=V(x',-x''),
\end{align}
where $x=(x',x'')\in\mathbb{R}^2\times\mathbb{R}^{N-2}$. We assume this henceforth.

\medskip
We shall briefly describe the solutions to be constructed later and will give the main ideas in the procedure of the construction. In particular, we shall introduce the {\em intermediate reduction method}, which we believe  can be very useful in other contexts.

 Firstly, without loss of generality, we can assume that $V_\infty=1$ by suitable scaling. As developed in \cite{WY-10}, we will use the loss of compactness to build up solutions. More precisely, we will construct solutions with large number of spikes whose inter-distances and distances from the origin are sufficiently large.

By the asymptotic behaviour of $V$ at infinity, the basic building block is the ground state (radial) solution $w$ of the limit problem at infinity:
\begin{align}\label{eq-w}
\left\{
\begin{array}{ll}
-\Delta w+w-w^p=0,\  w>0 \ \text{in}\ \mathbb{R}^N,
\vspace{1mm}\\
w=w(|x|), \ \ \ w\in H^1(\mathbb{R}^N).
\end{array}\right.
\end{align}
The solutions we construct will be small perturbations of the sum of copies of $w$, centered at some carefully chosen points on $\mathbb{R}^2\times\{0\}\subset\mathbb{R}^N$, where $0$ is the zero vector in $\mathbb{R}^{N-2}$.

Let $K\in\mathbb{N}_+$ be the number of spikes, whose locations are given by $Q_j\in\mathbb{R}^N$, $j=1,\dots,K$. We define
\begin{equation}\label{def-U}
w_{Q_j}(x)=w(x-Q_j)\ \text{and}\ U(x)=\sum_{j=1}^Kw_{Q_j}(x),\ \text{for}\ x\in\mathbb{R}^N.
\end{equation}
A natural and central question is how to choose $Q_j$'s such that a small perturbation of $U$ will be a genuine solution.

Assuming that
\[
\inf_{1\leq j\leq K}|Q_j|\rightarrow\infty\ \text{and}\
\inf_{j\neq l}|Q_j-Q_l|\rightarrow\infty,
\]
by the asymptotic behaviour of $V$ at infinity and the property of $w$, one can get (at least formally) the following energy expansion
\begin{align}\label{energy-w}
\mathcal{E}(U)=\underbrace{KI_0+a_0\sum_{j=1}^K|Q_j|^{-m}-\frac{1}{2}\gamma_0\sum_{j\neq l}w(|Q_j-Q_l|)}_{J(Q_1,\dots,Q_K)}+\text{other terms},
\end{align}
where $I_0$, $a_0$ and $\gamma_0$ are  positive constants. Here $\mathcal{E} (U)$ is the energy functional defined at (\ref{energy}) and we denote the leading order expansion as $J(Q_1,..., Q_K)$.

Observe that for any rotation $R_\theta$ around the origin in $\mathbb{R}^N$, there holds
\begin{equation*}
J(R_\theta Q_1,\dots,R_\theta Q_K)=J(Q_1,\dots,Q_K).
\end{equation*}
Hence any critical point of $J(Q_1,\dots,Q_K)$ is {\em degenerate}. Therefore, except  in the symmetric class, it is not easy to find critical points of small perturbations of $J(Q_1,\dots,Q_K)$. This means that it is not easy to apply the localized energy method directly. However, this observation gives us some enlightenment in the non-symmetric setting. Actually, when we restrict $Q_j$'s on a plane, this suggests us to introduce one more parameter to deal with the degeneracy due to rotations as we will see in Section 5.

Under the condition \eqref{cond-even-1}, there is no essential difference between $N=2$ and $N\geq3$. Hence from now on we will restrict $Q_j$'s on the plane $\mathbb{R}^2\times\{0\}\subset\mathbb{R}^N$. To describe further the configuration space of $Q_j$'s, we define
\begin{equation*}
Q_j^0=\left(R\cos\theta_j,R\sin\theta_j,0\right)\in\mathbb{R}^2\times\{0\},\ \text{for}\ j=1,\dots,K,
\end{equation*}
where
\begin{align*}
\theta_j=\alpha+(j-1)\frac{2\pi}{K}\in\mathbb{R}.
\end{align*}
Here $\alpha$ is the parameter representing  the degeneracy due to rotations, and $R$ is a positive constant to be determined later. Observe that each point $Q_j^0$ depends on $\alpha$. Thus we write $Q_j^0=Q_j^0(\alpha)$. When $\alpha=0$, the $Q_j^0$'s are the points used in \cite{WY-10}. If $V(x)$ is radially symmetric, it is obvious that the parameter $\alpha$ plays no role in the construction in \cite{WY-10}. But it is very important in our construction as we will see in Section 6.

For the constant $R$, we introduce the so-called balancing condition:
\begin{align}\label{eq-R}
a_0mR^{-m-1}=2\sin\frac{\pi}{K}\Psi\big(2R\sin\frac{\pi}{K}\big),
\end{align}
where $a_0=\frac{a}{2}\int_{\mathbb{R}^N}w^2\,dx>0$,
and $\Psi$ is the interaction function
defined by
\begin{equation}\label{def-Psi}
\Psi(s)=-\int_{\mathbb{R}^N}w(x-s\vec{\textbf{e}})\,\text{div}\big(w^p(x)\vec{\textbf{e}}\big)dx.
\end{equation}
Here $\vec{\textbf{e}}$ can be any unit vector in $\mathbb{R}^N$ (cf. \cite{M-09,MPW-12}). The balancing condition \eqref{eq-R} can either be understood as a consequence of a conservation law or can be seen as a condition such that the approximation $U$ is very close to a genuine solution
(cf. Appendix in \cite{MPW-12}). Assuming that
\begin{equation*}
d=2R\sin\frac{\pi}{K}\rightarrow+\infty,\ \text{as}\ K\rightarrow+\infty,
\end{equation*}
we will see that (cf. Lemma~\ref{lem-d-R})
\begin{align*}
|Q_j^0|=R\sim\frac{m}{2\pi}K\ln K,\ \text{and}\ \min_{j\neq l}\{|Q_j^0-Q_l^0|\}=d\sim m\ln K.
\end{align*}

Next we define a small neighbourhood of $\textbf{Q}^0=(Q_1^0,\dots,Q_K^0)$ on $(\mathbb{R}^2\times\{0\})^{K}$ in a suitable norm to be made precise and introduce another parameter. Let $f_j,\,g_j\in\mathbb{R}$, $j=1,\dots,K$, we define
\begin{align}\label{eq-Q}
Q_j=Q_j^0+f_j\vec{n}_j+g_j\vec{t}_j=(R+f_j)\vec{n}_j+g_j\vec{t}_j,
\end{align}
where
\begin{align*}
\vec{n}_j=\left(\cos\theta_j,\sin\theta_j,0\right),\ \text{and}\
\vec{t}_j=\left(-\sin\theta_j,\cos\theta_j,0\right).
\end{align*}

Keep in mind that $f_j$ and $g_j$  measure the displacement in the normal and tangential directions respectively.

Writing  $Q_j=Q_j(\alpha)$, $\vec{n}_j=\vec{n}_j(\alpha)$ and $\vec{t}_j=\vec{t}_j(\alpha)$, we note the following  trivial but  important fact:
\begin{equation}\label{Q-1}
Q_j(\alpha+2\pi)=Q_j(\alpha),\ \forall\,\alpha\in\mathbb{R},\ \text{and}\ \forall\,j=1,\dots,K.
\end{equation}

We can now introduce another parameter $\textbf{q}$ and define a suitable norm. Denote
\begin{align*}
\textbf{q}=(f_1,\cdots,f_K,g_1,\cdots,g_K)^T\in\mathbb{R}^{2K},
\end{align*}
that is, $q_j=f_j$ and $q_{K+j}=g_j$ for $j=1,\dots,K$. We define
\begin{align*}
\dot{\textbf{q}}=(\dot{f}_1,\cdots,\dot{f}_K,\dot{g}_1,\cdots,\dot{g}_K)^T,\ \text{and}\
\ddot{\textbf{q}}=(\ddot{f}_1,\cdots,\ddot{f}_K,\ddot{g}_1,\cdots,\ddot{g}_K)^T,
\end{align*}
where for $j=1,\dots,K$,
\begin{align*}
&\dot{f}_j=(f_{j+1}-f_j)\frac{K}{2\pi},\
\ddot{f}_j=(f_{j+1}-2f_j+f_{j-1})\frac{K^2}{4\pi^2},\\
&\dot{g}_j=(g_{j+1}-g_j)\frac{K}{2\pi},\
\ddot{g}_j=(g_{j+1}-2g_j+g_{j-1})\frac{K^2}{4\pi^2},\\
&f_{K+1}=f_1,\ f_0=f_K,\ g_{K+1}=g_1,\ g_0=g_K.
\end{align*}
Observe that if $f_j=f(\theta_j)$ for some $2\pi$ periodic smooth function $f$, then $\dot{f}_j$ is the forward difference of $f$ and $\ddot{f}_j$ is the 2nd order central difference of $f$.

\medskip

With these notation, we can define the configuration space of $Q_j$'s by
\begin{align*}
\Lambda_K=\Big\{(Q_1,\dots,Q_K)\in(\mathbb{R}^2\times\{0\})^K\,\big|\,Q_j\ \text{is defined by}\ \eqref{eq-Q}\ \text{and}\ \|\textbf{q}\|_*\leq1\Big\},
\end{align*}
where $\|\textbf{q}\|_*=\|\textbf{q}\|_\infty+\|\dot{\textbf{q}}\|_\infty+\|\ddot{\textbf{q}}\|_\infty$ is a norm on $\mathbb{R}^{2K}$. In the following, we assume that $Q_j$ is defined by \eqref{eq-Q}, the parameter $\alpha\in\mathbb{R}$ and the parameter $\textbf{q}$ satisfy
\begin{align}\label{cond-q}
\|\textbf{q}\|_*=\|\textbf{q}\|_\infty+\|\dot{\textbf{q}}\|_\infty+\|\ddot{\textbf{q}}\|_\infty\leq 1.
\end{align}
For any $(Q_1,\dots,Q_K)\in\Lambda_K$, an easy computation shows that for $j=1,\dots,K$,
\begin{align*}
|Q_j|=R+f_j+O(R^{-1}),
\end{align*}
and
\begin{align*}
|Q_{j+1}-Q_j|=d+2(f_j+\dot{g}_j)\frac{\pi}{K}+O(K^{-2}).
\end{align*}
Define $\rho=\min\limits_{j\neq l}\{|Q_j-Q_l|\}$, it follows that
\begin{align}\label{eq-rho}
\rho=d+O(K^{-1}),\ \text{and}\ \min_{j=1,\dots,K}\{|Q_j|\}=R+O(1).
\end{align}

We will prove Theorem \ref{thm-0} by showing the following result.

\begin{theorem}\label{thm-2}
Under the assumption of Theorem \ref{thm-0}, there is a positive integer $K_0$ such that: for all integer $K\geq K_0$, there exist $\alpha\in[0,2\pi)$ and $(Q_1,\dots,Q_K)\in\Lambda_K$ such that problem \eqref{eq-u} has {\bf two  solutions} of the form
\begin{equation}\label{u_K}
u(x)=\sum_{j=1}^Kw(x-Q_j)+\phi(x),
\end{equation}
where
$\|\phi\|_{H^1(\mathbb{R}^N)}+\|\phi\|_{L^\infty(\mathbb{R}^N)}\rightarrow0$ as $K\rightarrow+\infty$.
Moreover, the energy of $u$ is given by
\begin{align}
\mathcal{E}(u)=K\Big(\frac{1}{2}-\frac{1}{p+1}\Big)\int_{\mathbb{R}^N}w^{p+1}\,dx+o(1).
\end{align}
\end{theorem}

\begin{remark}
It is worth pointing out that the solutions constructed in this paper are different from those found in \cite{CDS-11,AW-12}. The reason is simply that the inter-distances and distances from the origin of the spikes of the solutions given in \eqref{u_K} tend to infinity uniformly as $K$ goes to infinity, but those of the solutions found in \cite{CDS-11,AW-12} do not.
\end{remark}

\begin{remark} The fact that we can find at least {\bf two} solutions of the form (\ref{u_K}) is nontrivial. This is due to the fact that we need to choose the {\bf first} starting point $Q_1^0$. It turns out that there are at least two such points to choose (see Section 6).
\end{remark}

\begin{remark} As $ K\to +\infty$, $(f_j, g_j)$ is the discretization of two second order ordinary differential equations (\ref{cont}).
\end{remark}

To prove Theorem \ref{thm-2}, it is sufficient to show that for $K$ sufficiently large there are parameters $\alpha$ and $\textbf{q}$ such that $U+\phi$ is a genuine solution for a small perturbation $\phi$. To achieve this goal, we will adopt the techniques in the singularly perturbed problem. Unlike problem \eqref{eq-u-h}, there is no apparent parameter in \eqref{eq-u}. As stated in Theorem \ref{thm-2}, we use the number of the spikes as the $\varepsilon$ type parameter. This idea comes directly from \cite{WY-10} and goes back at least as far as to \cite{LNW-07}.

Before we sketch the procedure of our proof, we briefly introduce the abstract set-up of the Lyapunov-Schmidt reduction (although it is always used in a framework that occurs often in bifurcation theory).

Let $X,Y$ be Banach spaces and $S(u)$ is a $C^1$ map from $X$ to $Y$. To study the equation $S(u)=0$, a natural way is to find approximations first and then to look for genuine solutions as (small) perturbations of approximations.
Assume that $U_\lambda$ are the approximations, where $\lambda\in\Lambda$ is the parameter (we think of $\Lambda$ as the configuration space). Writing $u=U_\lambda+\phi$, then solving $S(u)=0$ amounts to solve
\begin{equation}\label{eq-L-S}
L[\phi]+E+N(\phi)=0,
\end{equation}
where
\[
L[\phi]=S'(U_\lambda)[\phi],\ E=S(U_\lambda),\ \text{and}\ N(\phi)=S(U_\lambda+\phi)-S(U_\lambda)-S'(U_\lambda)[\phi].
\]
Here $S'(U_\lambda)$ is the Fr\'{e}chet derivative of $S$ at $U_\lambda$, $E$ denotes the error of approximation, and $N(\phi)$ denotes the nonlinear term. In order to solve \eqref{eq-L-S}, we try to invert the linear operator $L$ so that we can rephrase the problem as a fixed point problem. That is, when $L$ has a uniformly bounded inverse in a suitable space, one can rewrite the equation \eqref{eq-L-S} as
\[
\phi=-L^{-1}[E+N(\phi)]=\mathcal{A}(\phi).
\]
What is left is to use fixed point theorems such as contraction mapping theorem.

The Lyapunov-Schmidt reduction deals with the situation when the linear operator $L$ is Fredholm and its eigenfunction space associated to small eigenvalues has finite dimensional.
Assuming that $\{\mathcal{Z}_1,\dots,\mathcal{Z}_n\}$ is a basis of the eigenfunction space associated to small eigenvalues of $L$, we can divide the procedure of solving \eqref{eq-L-S} into two steps:
\begin{enumerate}[(i)]
\item solving the projected problem for any $\lambda\in\Lambda$,
\begin{equation*}
\begin{cases}
L[\phi]+E+N(\phi)=\sum\limits_{j=1}^nc_j\mathcal{Z}_j,\\
\langle\phi,\mathcal{Z}_{j}\rangle=0,\ \forall\,j=1,\dots,n,
\end{cases}
\end{equation*}
where $c_j$ may be constant or function depending on the form of $\langle\phi,\mathcal{Z}_{j}\rangle$.
\item solving the reduced problem
\[
c_j(\lambda)=0,\ \forall\,j=1,\dots,n,
\]
by adjusting $\lambda$.
\end{enumerate}

Let us now turn to our problem \eqref{eq-u}. In this case,
\begin{align*}
S(u)&=-\Delta u+V(x) u-u_+^p,\\
L[\phi]&=-\Delta\phi+V(x)\phi-pU^{p-1}\phi,\\
E&=-\Delta U+V(x)U-U^p,\\
N(\phi)&=-(U+\phi)_+^p+U^p+pU^{p-1}\phi.
\end{align*}
Observe that all of these quantities depend implicitly on $\alpha$ and $\textbf{q}$ even though this is not apparent in the notation.

By the Lyapunov-Schmidt reduction, the procedure of construction is made up of several steps which we explain next and postpone the proofs of major facts in later sections.

\medskip

\noindent
\textbf{Step 1: Solving the projected problem.}

\vspace{1mm}

Let $\alpha\in\mathbb{R}$ and $\textbf{q}$ satisfy \eqref{cond-q}. We look for a function $\phi$ and some multiplier $\widehat{\beta}\in\mathbb{R}^{2K}$ such that
\begin{equation}\label{eq-u-1-a}
\begin{cases}
L[\phi]+E+N(\phi)=\widehat{\beta}\cdot\frac{\partial U}{\partial\textbf{q}},
\vspace{1mm}\\
\int_{\mathbb{R}^N}\phi\,\mathcal{Z}_{Q_j}\,dx=0,\ \forall\,j=1,\dots,K,
\end{cases}
\end{equation}
where the vector field $\mathcal{Z}_{Q_j}$ is defined by
\begin{equation}\label{eq-Z}
\mathcal{Z}_{Q_j}(x)
=\nabla w(x-Q_j).
\end{equation}
By direct computation, we have
\begin{align*}
\frac{\partial U}{\partial\textbf{q}}
&=-\big(\mathcal{Z}_{Q_1}\cdot\vec{n}_1,\cdots,\mathcal{Z}_{Q_K}\cdot\vec{n}_K,\mathcal{Z}_{Q_1}\cdot\vec{t}_1,\cdots,\mathcal{Z}_{Q_K}\cdot\vec{t}_K\big)^T.
\end{align*}

This is the first step in the Lyapunov-Schmidt reduction. It is done in Section 4 through some a priori estimates and contraction mapping theorem. A required element in this step is the non-degeneracy of $w$ (cf. Lemma \ref{lem-w}).
It is worth pointing out that the function $\phi$ and the multiplier $\widehat{\beta}$ found in Step 1 depend on the parameters $\alpha$ and $\textbf{q}$. Hence we write $\phi=\phi(x;\alpha,\textbf{q})$ and $\widehat{\beta}=\widehat{\beta}(\alpha,\textbf{q})$.

\medskip

\noindent
\textbf{Step 2: Solving the reduced problem}

\vspace{1mm}

By Step 1, it is known that $\widehat{\beta}$ is small. But it is not easy to solve $\widehat{\beta}(\alpha,\textbf{q})=0$ directly since the linear part of the expansion of $\widehat{\beta}$ in $\textbf{q}$ is degenerate (due to the invariance of $J(Q_1,\dots,Q_K)$ under rotations).

More precisely, let us write
\begin{equation*}
\widehat{\beta}(\alpha,\textbf{q})=\widetilde{T}\textbf{q}+\Phi(\alpha,\textbf{q}),
\end{equation*}
where $\widetilde{T}\textbf{q}$ is the linear part and $\Phi(\alpha,\textbf{q})$ denotes the remaining term.
As we will see in Section 5, $\widetilde{T}\textbf{q}$ does not depend on $\alpha$ and there is a unique vector (up to a scalar)
\begin{align*}
\textbf{q}_0=(\underbrace{0,\dots,0}_{K},\underbrace{1,\dots,1}_{K})^T\in\mathbb{R}^{2K}
\end{align*}
such that $\widetilde{T}\textbf{q}_0=0$.

By the Lyapunov-Schmidt reduction again (called the secondary Lyapunov-Schmidt reduction), the step of solving the reduced problem $\widehat{\beta}(\alpha,\textbf{q})=0$ can be divided into two steps.
To write the projected problem of $\widehat{\beta}=0$ in a proper form, note that
\begin{align*}
\frac{\partial U}{\partial \alpha}&=R\sum_{j=1}^K\frac{\partial U}{\partial g_j}+\sum_{j=1}^K\left(f_j\frac{\partial U}{\partial g_j}-g_j\frac{\partial U}{\partial f_j}\right)=(R\textbf{q}_0+\textbf{q}^\perp)\cdot\frac{\partial U}{\partial\textbf{q}},
\end{align*}
where $\textbf{q}^\perp=(-\vec{g},\vec{f})$ for $\textbf{q}=(\vec{f},\vec{g})$. Hence we define
\begin{equation}\label{beta-gamma}
\vec{\beta}=\widehat{\beta}-\gamma(R\textbf{q}_0+\textbf{q}^\perp),\ \text{for every }\gamma\in\mathbb{R}.
\end{equation}
Obviously the new multiplier $\vec{\beta}$ depends on the parameters $\alpha$, $\textbf{q}$ and $\gamma$. Thus we write $\vec{\beta}=\vec{\beta}(\alpha,\textbf{q},\gamma)$.

\medskip

\noindent
\textbf{Step 2.A: Solving $\vec{\beta}(\alpha,\textbf{q},\gamma)=0$ by adjusting $\gamma$ and $\textbf{q}$.}

\vspace{1mm}

In this step, for each $\alpha\in\mathbb{R}$, we are going to find parameters $(\gamma,\textbf{q})$ such that
\begin{equation}\label{solve-beta}
\vec{\beta}(\alpha,\textbf{q},\gamma)=0,\  \text{and}\ \textbf{q}\perp\textbf{q}_0.
\end{equation}
It can be seen as the step of solving the projected problem in the secondary Lyapunov-Schmidt reduction. To achieve it, we will use the condition \eqref{cond-m}. This step is done in Section 5 by using various integral estimates and contraction mapping theorem. A key element in this step is the invertibility of an $2K\times2K$ matrix whose proof is given in Appendix A.

When Step 2.A is done, we denote the unique solution of \eqref{solve-beta} by $(\gamma(\alpha),\textbf{q}(\alpha))$. Then the original problem~\eqref{eq-u} is reduced to the problem $\gamma(\alpha)=0$ of one dimension.

\medskip

\noindent
\textbf{Step 2.B: Solving $\gamma(\alpha)=0$ by choosing $\alpha$.}

\vspace{1mm}

At the last step, we want to prove that there exists an $\alpha$ such that $\gamma(\alpha)=0$. As a result, the function $u=U+\phi$ is a genuine solution of \eqref{eq-u}.

This step is the second step of solving the reduced problem in the secondary Lyapunov-Schmidt reduction. To achieve this step, by Step 2.A, the function $\phi=\phi(x;\alpha,\textbf{q}(\alpha))$ found in Step 1 solves the following problem:
\begin{equation}\label{eq-u-2}
\begin{cases}
L[\phi]+E+N(\phi)=\gamma(\alpha)\frac{\partial U}{\partial\alpha},
\vspace{1mm}\\
\int_{\mathbb{R}^N}\phi\,\mathcal{Z}_{Q_j}\,dx=0,\ \forall\,j=1,\dots,K,
\end{cases}
\end{equation}
where all of the quantities depending implicitly on $(\alpha,\textbf{q})$ are taken values at $(\alpha,\textbf{q}(\alpha))$.
To solve $\gamma(\alpha)=0$, we first apply the so-called variational reduction (often used in the localized energy method) to show that equation $\gamma(\alpha)=0$ has a solution if the reduced energy function $F(\alpha)=\mathcal{E}(U+\phi)$ has a critical point. Secondly, by using \eqref{Q-1}, it is easy to check that $F(\alpha)$ is $2\pi$ periodic in $\alpha$. Hence it has at least {\bf two} critical points. More details of this step will be given in Section 6.

\vspace{3mm}

Finally, this paper is organized as follows. Some preliminary facts and estimates are explained in Section 3. In Section 4 we apply the standard Lyapunov-Schmidt reduction for Step 1.
Section 5 contains a further reduction process for Step 2.A which reduces the original problem to one dimension. In Section 6 we carry out Step 2.B and then complete the proof of Theorem \ref{thm-2}. At the last, we discuss some possible extensions in Section 7.


\section{Preliminaries}

In this section we present some preliminary facts and some useful estimates.

First we recall some basic and useful properties of the standard spike solution $w$ defined by \eqref{eq-w} and those of the interaction function $\Psi$ defined in \eqref{def-Psi}.

\begin{lemma}\label{lem-w}
If $1<p<\infty$ for $N=2$ and $1<p<\frac{N+2}{N-2}$ for $N\geq3$, then every positive solution of the problem:
\begin{align}
\left\{
\begin{array}{ll}
-\Delta u+u-u^p=0\ \text{in}\ \mathbb{R}^N,
\vspace{1mm}\\
u>0\ \text{in}\ \mathbb{R}^N,\ u\in H^1(\mathbb{R}^N),
\end{array}\right.
\end{align}has the form $w(\cdot-Q)$ for some $Q\in\mathbb{R}^N$, where $w(x)=w(|x|)\in C^\infty(\mathbb{R}^N)$ is the unique positive radial solution which satisfies
\begin{equation}
\lim_{r\rightarrow\infty}r^{\frac{N-1}{2}}e^rw(r)=c_{N,p},\quad
\lim_{r\rightarrow\infty}\frac{w'(r)}{w(r)}=-1.
\end{equation}
Here $c_{N,p}$ is a positive constant depending only on $N$ and $p$. Furthermore, the Morse index of $w$ is one and $w$ is nondegenerate in the sense that
\[
\text{Ker}\left(-\Delta+1-pw^{p-1}\right)\cap L^\infty(\mathbb{R}^N)=\text{Span}\left\{\partial_{x_1}w,\cdots,\partial_{x_N}w\right\}.
\]
\end{lemma}

\begin{proof}
This result is well known. For the proof we refer the reader to \cite{BL-83} for the existence, \cite{GNN-81} for the symmetry, \cite{K-89} for the uniqueness, Appendix C in \cite{NT-93} for the nondegeneracy, and \cite{BW-10} for the Morse index.
\end{proof}

\begin{lemma}\label{lem-Psi}
For $s$ sufficiently large,
\begin{equation}\label{Psi}
\Psi(s)=c_{N,p}\,s^{-\frac{N-1}{2}}e^{-s}\left(1+O(s^{-1})\right).
\end{equation}
where $c_{N,p}>0$ is a constant depending only on $N$ and $p$.
\end{lemma}

\begin{proof}
This lemma follows from Taylor's theorem and the Lebesgue dominated convergence theorem. We omit it here and refer to \cite{KW-00,M-09} for details.
\end{proof}

Next we study the balancing condition \eqref{eq-R}. Assuming that
\begin{equation*}
d=2R\sin\frac{\pi}{K}\rightarrow+\infty,\ \text{as}\ K\rightarrow+\infty,
\end{equation*}
by the expansion \eqref{Psi}, both positive numbers $R$ and $d$ are uniquely determined by $K$. Moreover, we have the following expansions.

\begin{lemma}\label{lem-d-R}
For $K$ sufficiently large,
\begin{align}\label{eq-d}
d&=m\ln K+\Big(m-\frac{N-3}{2}\Big)\ln(m\ln K)+O(1),\\
R&=\frac{m}{2\pi}K\ln K+\frac{1}{2\pi}\Big(m-\frac{N-3}{2}\Big)K\ln(m\ln K)+O(K).\nonumber
\end{align}
\end{lemma}

\begin{proof}
From the balancing conditon\eqref{eq-R}, the number $d$ satisfies the equation
\begin{align}\label{eq-d-1}
d^{m+1}\Psi(d)=a_0m\big(2\sin\frac{\pi}{K}\big)^m.
\end{align}
By using \eqref{Psi}, for $K$ sufficiently large, equation \eqref{eq-d-1} becomes
\begin{align*}
c_{N,p}d^{m+1-\frac{N-1}{2}}e^{-d}(1+O(d^{-1}))=a_0m(2\pi)^mK^{-m}\big(1+O(K^{-2})\big),
\end{align*}
from which we obtain
\[
d^{m+1-\frac{N-1}{2}}e^{-d}\sim K^{-m}.
\]
Let $d=m\ln K+d_1$ with $d_1=o(\ln K)$, we have
\begin{align*}
c_{N,p}(m\ln K)^{m+1-\frac{N-1}{2}}e^{-d_1}\big(1+o(1)\big)
=a_0m(2\pi)^m(1+O(K^{-2})).
\end{align*}
It follows that
\[
e^{d_1}\sim(m\ln K)^{m+1-\frac{N-1}{2}}.
\]
Therefore,
\begin{align*}
d_1=\Big(m-\frac{N-3}{2}\Big)\ln(m\ln K)+O(1),
\end{align*}
from which we get the expansions of $d$ and $R$.
\end{proof}

In the next section, we will apply the Lyapunov-Schmidt reduction. After refinements  by many authors working on the subject or on closely related problems, this type of argument is rather standard now. However technical difficulties arise when the number of spikes goes to infinity or the number of spikes is infinity (cf. \cite{M-09}). To deal with these difficulties, the following lemmas are useful.

\begin{lemma}\label{lem-N}
There exists a constant $C_N$ depending only on $N$ such that for any $K\in\mathbb{N}_+$ and any $\textbf{Q}=(Q_1,\dots,Q_K)\in(\mathbb{R}^N)^K$,
\begin{equation}\label{Q-l-N}
\#\Big\{Q_j\,\big|\,\ell\rho/2\leq|Q_j-x|<(\ell+1)\rho/2\Big\}\leq C_N(\ell+1)^{N-1},
\end{equation}
for all $x\in\mathbb{R}^N$ and all $\ell\in\mathbb{N}$, where $\rho=\min_{j\neq l}\{|Q_j-Q_l|\}$.
In particular, if $\textbf{Q}=(Q_1,\dots,Q_K)\in(\mathbb{R}^2\times\{0\})^K$, then for all $x\in\mathbb{R}^N$ and all $\ell\in\mathbb{N}$,
\begin{equation}\label{Q-l-2}
\#\Big\{Q_j\,\big|\,\ell\rho/2\leq|Q_j-x|<(\ell+1)\rho/2\Big\}\leq6(\ell+1).
\end{equation}
\end{lemma}

\begin{proof}
When $\rho=0$, the result is trivial. It remains to consider the case $\rho>0$.
For $\ell=0$, it suffices to take $C_N\geq1$. For $\ell\geq1$, let $Q_{j_k}$'s ($k=1,\dots,n$) be the points satisfying
\begin{equation*}
\ell\rho/2\leq|Q_{j_k}-x|<(\ell+1)\rho/2.
\end{equation*}
By the triangle inequality, we have
\begin{equation*}
(\ell-1)\rho/2\leq|y-x|<(\ell+2)\rho/2,\ \forall\,y\in B_{\rho/2}(Q_{j_k}).
\end{equation*}
Hence for all $k=1,\dots,n$,
\begin{equation*}
B_{\rho/2}(Q_{j_k})\subset B_{(\ell+2)\rho/2}(x)\setminus B_{(\ell-1)\rho/2}(x).
\end{equation*}
Since $B_{\rho/2}(Q_{j_l})\cap B_{\rho/2}(Q_{j_k})=\emptyset$ for $l\neq k$,  we conclude that
\begin{equation*}
\sum_{k=1}^n\left|B_{\rho/2}(Q_{j_k})\right|\leq\left|B_{(\ell+2)\rho/2}(x)\setminus B_{(\ell-1)\rho/2}(x)\right|.
\end{equation*}
Therefore, taking $C_N=\sup_{\ell\in\mathbb{N}_+}\frac{(\ell+2)^N-(\ell-1)^N}{(\ell+1)^{N-1}}$, we have
\begin{align*}
n\leq (\ell+2)^N-(\ell-1)^N\leq C_N(\ell+1)^{N-1},
\end{align*}
which implies \eqref{Q-l-N}.

If $\textbf{Q}=(Q_1,\dots,Q_K)\in(\mathbb{R}^2\times\{0\})^K$, the above argument implies that
\[
\sum_{k=1}^n\left|B_{\rho/2}(Q_{j_k})\cap\mathbb{R}^2\times\{0\}\right|\leq\big|B_{(\ell+2)\rho/2}(x)\setminus B_{(\ell-1)\rho/2}(x)\cap\mathbb{R}^2\times\{0\}\big|,
\]
which implies that
\[
n\leq (\ell+2)^2-(\ell-1)^2\leq6(\ell+1).
\]
Therefore, we get the estimate~\eqref{Q-l-2} if we restrict $\textbf{Q}=(Q_1,\dots,Q_K)$ on $(\mathbb{R}^2\times\{0\})^K$.
\end{proof}

Given $\textbf{Q}=(Q_1,\dots,Q_K)\in(\mathbb{R}^N)^K$ with $\rho=\min_{j\neq l}\{|Q_j-Q_l|\}>0$, for any $\ell\in\mathbb{N}$, we divide $\mathbb{R}^N$ into $K+1$ parts:
\begin{equation*}
\Omega_j^\ell=\Big\{x\in\mathbb{R}^N\,\big|\,|x-Q_j|=\min_{1\leq l\leq K}|x-Q_l|\leq \ell\rho/2\Big\},\ \forall\,j=1,\dots,K,
\end{equation*}
and
$\Omega_{K+1}^\ell=\mathbb{R}^N\setminus\cup_{j=1}^K\Omega_j^\ell
$. Then the interior of $\Omega_j^\ell\cap\Omega_l^\ell$ is an empty set for $j\neq l$.

\begin{lemma}\label{lem-3.5}
Suppose that $\varGamma(r)$ is a positive decreasing function defined on $[0,\infty)$ such that for some $b\in\mathbb{R}$ and $\eta>0$,
\begin{equation}\label{gamma-infinity}
\varGamma(r)\sim r^{b}e^{-\eta r}\ \text{as}\ r\rightarrow\infty.
\end{equation}
Then there exist positive constants $\rho_0$ and $C$  (independent of $K$) such that
\begin{enumerate}[(i)]
\item for all $K,\ell\in\mathbb{N}_+$, all $(Q_1,\dots,Q_K)\in(\mathbb{R}^N)^K$ with $\rho\geq\rho_0$, and all $x\in\Omega_{j_0}^\ell$ ($j_0=1,\dots,K$), we have
\begin{equation}\label{eq-gamma-0}
\sum_{j=1}^K\varGamma(|x-Q_j|)\leq C\ell^{N-1}\varGamma(|x-Q_{j_0}|).
\end{equation}
In particular, if $(Q_1,\dots,Q_K)\in(\mathbb{R}^2\times\{0\})^K$, then
\begin{equation*}
\sum_{j=1}^K\varGamma(|x-Q_j|)\leq C\ell\varGamma(|x-Q_{j_0}|).
\end{equation*}

\item for all $(Q_1,\dots,Q_K)\in(\mathbb{R}^N)^K$ with $\rho\geq\rho_0$ and all $j_0\in\{1,\dots,K\}$,
\begin{equation}\label{eq-gamma}
\sum_{j\neq j_0}\varGamma(|Q_{j_0}-Q_{j}|)\leq C\varGamma(\rho).
\end{equation}
\end{enumerate}
\end{lemma}

\begin{remark}
A similar result holds when $\varGamma(r)$ has polynomial decay. For example, if for some integer $n\in\mathbb{N}_+$,
\begin{equation*}
\varGamma(r)\sim r^{b}\ \text{as}\ r\rightarrow+\infty,\ \text{where}\ b<-n,
\end{equation*}
then there are positive constants $\rho_0$ and $C$ (independent of $K$) such that for all $K\in\mathbb{N}_+$, all $(Q_1,\dots,Q_K)\in(\mathbb{R}^n\times\{0\})^K$ with $\rho\geq\rho_0$, and all $j_0\in\{1,\dots,K\}$,
\begin{equation*}
\sum_{j\neq j_0}\varGamma(|Q_{j}-Q_{j_0}|)\leq C\varGamma(\rho).
\end{equation*}
This kind of property is useful and important in the construction of infinitely many solutions of problem with critical growth.
\end{remark}

\begin{proof}
Given $x\in\Omega_{j_0}^\ell$, by definition we have
\[
|x-Q_{j_0}|\leq\ell\rho/2\ \text{and}\ |x-Q_{j_0}|\leq|x-Q_j|,\ \forall\,j=1,\dots,K.
\]
Thus there is an integer $0\leq\ell_0\leq \ell$ such that
\begin{equation*}
\ell_0\rho/2\leq|x-Q_{j_0}|<(\ell_0+1)\rho/2.
\end{equation*}
By the property of $\varGamma(r)$ and Lemma \ref{lem-N}, for $\rho$ sufficiently large, we have
\begin{align*}
\sum_{j=1}^K\varGamma(|x-Q_j|)&
\leq C_N(\ell_0+1)^{N-1}\varGamma(|x-Q_{j_0}|)+C_N\sum_{s=\ell_0+1}^{+\infty}(s+1)^{N-1}\varGamma(s\rho/2)\\
&\leq C_N(\ell_0+1)^{N-1}\varGamma(|x-Q_{j_0}|)+C(\ell_0+2)^{N-1}\varGamma\big((\ell_0+1)\,\rho/2\big)\\
&\leq C\ell^{N-1}\varGamma(|x-Q_{j_0}|),
\end{align*}
where in the second inequality we use the following inequality:
\begin{align*}
\sum_{s=\ell_0+1}^{+\infty}\frac{(s+1)^{N-1}}{(\ell_0+2)^{N-1}}\frac{\varGamma(s\rho/2)}{\varGamma\big((\ell_0+1)\rho/2\big)}
\leq C.
\end{align*}
To prove it, for $\rho$ sufficiently large, by \eqref{gamma-infinity}, we have
\begin{equation*}
\frac{\varGamma(s\rho/2)}{\varGamma\big((\ell_0+1)\rho/2\big)}
\leq C\Big(\frac{s}{\ell_0+1}\Big)^b e^{-\eta(s-\ell_0-1)\rho/2}.
\end{equation*}
Hence
\begin{align*}
\sum_{s=\ell_0+1}^{+\infty}\frac{(s+1)^{N-1}}{(\ell_0+2)^{N-1}}\frac{\varGamma(s\rho/2)}{\varGamma\big((\ell_0+1)\rho/2\big)}
&\leq C\sum_{s=\ell_0+1}^{+\infty}\Big(\frac{s}{\ell_0+1}\Big)^{N-1+b}e^{-\eta(s-\ell_0-1)\rho/2}\\
&\leq C\sum_{s=\ell_0+1}^{+\infty}e^{-\eta(s-\ell_0-1)\rho/4}\\
&\leq C\int_{0}^{+\infty}e^{-\eta t\rho/4}\,dt
\leq C.
\end{align*}
In particular, if $(Q_1,\dots,Q_K)\in(\mathbb{R}^2\times\{0\})^K$, then by \eqref{Q-l-2}, we can take $N=2$ in the above arguments.

To deduce \eqref{eq-gamma} from \eqref{eq-gamma-0}, denote
\[
\widehat{\textbf{Q}}=(Q_1,\dots,Q_{j_0-1},Q_{j_0+1},\dots,Q_K)\in(\mathbb{R}^N)^{K-1},
\]
and
\[
\widehat{\rho}=\min_{j\neq l}\Big\{|Q_j-Q_l|\,\big|\,j\neq j_0,\,l\neq j_0\Big\}\geq\rho.
\]
Take $j_1\in\{1,\dots,K\}$ such that
\[
|Q_{j_0}-Q_{j_1}|=\min_{l\neq j_0}\big\{|Q_{j_0}-Q_l|\big\},
\]
and choose $\ell\in\mathbb{N}_+$ satisfying
\[
(\ell-1)\widehat{\rho}/2<|Q_{j_0}-Q_{j_1}|\leq\ell\widehat{\rho}/2.
\]
Then by \eqref{eq-gamma-0}, we have
\begin{align*}
\sum_{j\neq j_0}\varGamma(|Q_{j_0}-Q_{j}|)
&\leq C\Big(1+\frac{2|Q_{j_0}-Q_{j_1}|}{\widehat{\rho}}\Big)^{N-1}\varGamma(|Q_{j_0}-Q_{j_1}|)\\
&\leq C\varGamma(|Q_{j_0}-Q_{j_1}|)+C\varGamma(\widehat{\rho})\leq C\varGamma(\rho),
\end{align*}
where in the second inequality we use the following inequality:
\[
|Q_{j_0}-Q_{j_1}|^{N-1}\varGamma\big(|Q_{j_0}-Q_{j_1}|\big)
\leq C\widehat{\rho}^{N-1}\varGamma\big(\widehat{\rho}\big),\ \text{for}\ |Q_{j_0}-Q_{j_1}|\geq\widehat{\rho}\geq\rho_0.
\]
To prove it, we only need to apply \eqref{gamma-infinity}.
\end{proof}

A simple corollary is the following result which is useful in our construction.

\begin{corollary}\label{lem-u}
There are positive constants $\rho_0$ and $C$  (independent of $K$) such that for all $K,\ell\in\mathbb{N}_+$, all $(Q_1,\dots,Q_K)\in(\mathbb{R}^2\times\{0\})^K$ with $\rho\geq\rho_0$, and all $x\in\Omega_{j_0}^\ell$ ($j_0=1,\dots,K$), we have
\begin{equation}\label{ineq-q}
\sum_{j=1}^Kw_{Q_j}(x)\leq C\ell w_{Q_{j_0}}(x),
\end{equation}
and
\begin{align}\label{q-q-K}
\sum_{j\neq j_0}e^{-\eta|Q_j-Q_{j_0}|}\leq Ce^{-\eta\rho}.
\end{align}
\end{corollary}

To analyze the interactions between spikes, we prove some estimates concerning convolution of functions with suitable exponential decays.

\begin{lemma}\label{lem-acr-07}
Given $\varGamma_1,\varGamma_2$ two positive continuous radial functions on $\mathbb{R}^N$ with the following property:
\begin{align*}
\varGamma_1(r)\sim r^{b_1}e^{-\eta_1r},\ \text{and}\ \varGamma_2(r)\sim r^{b_2}e^{-\eta_2r},\ \text{as}\ r\rightarrow\infty,
\end{align*}
where $b_1,b_2\in\mathbb{R}$, $\eta_1>0$, $\eta_2>0$. Let $\xi\in\mathbb{R}^N$ tends to infinity. Then, the following asymptotic estimates hold:
\begin{enumerate}[(i)]
\item If $\eta_1<\eta_2$, then
\begin{equation*}
\int_{\mathbb{R}^N}\varGamma_1(x-\xi)\varGamma_2(x)\,dx\sim|\xi|^{b_1}e^{-\eta_1|\xi|}.
\end{equation*}
Clearly, if $\eta_1>\eta_2$, a similar expression holds, by replacing $b_1$ and $\eta_1$ with $b_2$ and $\eta_2$.
\item If $\eta_1=\eta_2$, suppose that $b_1\geq b_2$ for simplicity. Then
\begin{align*}
\int_{\mathbb{R}^N}\varGamma_1(x-\xi)\varGamma_2(x)\,dx\sim
\left\{\begin{array}{lll}
|\xi|^{b_1+b_2+\frac{N+1}{2}}e^{-\eta_1|\xi|},\quad&\text{if}\ b_2>-\frac{N+1}{2},
\vspace{2mm}\\
\big(|\xi|^{b_1}\ln|\xi|\big)e^{-\eta_1|\xi|},\quad&\text{if}\ b_2=-\frac{N+1}{2},
\vspace{2mm}\\
|\xi|^{b_1}e^{-\eta_1|\xi|},\quad&\text{if}\ b_2<-\frac{N+1}{2}.
\end{array}
\right.
\end{align*}
\end{enumerate}
\end{lemma}

\begin{proof}
This result follows from the Lebesgue dominated convergence theorem. The argument is standard and is omitted here, we refer the reader to Lemma 3.7 in \cite{ACR-07} for details.
\end{proof}

By the property of $w$, as a corollary of Lemma \ref{lem-acr-07}, we have the following integral estimates.

\begin{lemma}\label{lem-act}
Suppose that $|Q_j-Q_k|$ is sufficiently large, then the following estimates hold:
\begin{enumerate}[(i)]
\item for every $p>1$,
\begin{equation*}
\int_{\mathbb{R}^N}w_{Q_j}w_{Q_k}^p\,dx=(\gamma_0+o(1))w(|Q_j-Q_k|),
\end{equation*}
where
$\gamma_0=\int_{\mathbb{R}^N}w^p(x)e^{-x_1}\,dx>0$ is a constant;

\item
\begin{equation*}
\int_{\mathbb{R}^N}w_{Q_j}w_{Q_k}\,dx=O\big(e^{-|Q_k-Q_j|}|Q_k-Q_j|^{-(N-3)/2}\big);
\end{equation*}

\item let $\Omega_k=\big\{x\in\mathbb{R}^N\,\big|\,|x-Q_k|=\min\limits_{1\leq j\leq K}|x-Q_j|\big\}$, then
\begin{equation*}
\int_{\Omega_k}w_{Q_j}^pw_{Q_k}\,dx=O\big(e^{-\frac{p+1}{2}|Q_j-Q_k|}|Q_j-Q_k|^{-\frac{N-3}{2}}\big),
\end{equation*}
and
\begin{equation*}
\int_{\Omega_k}w_{Q_j}^2w_{Q_k}^{p-1}\,dx=O\big(e^{-\min\{2,\,\frac{p+1}{2}\}|Q_j-Q_k|}|Q_j-Q_k|^{-\frac{N-3}{2}}\big).
\end{equation*}
\end{enumerate}
\end{lemma}

\begin{proof}
Since the argument of proof is somewhat standard, we give only the main ideas of the proof.
\begin{enumerate}[(i)]
\item It follows from Lemma \ref{lem-w} and Lebegue's dominated convergence theorem (see e.g. the arguments used in the proof of Lemma 2.5, in \cite{LNW-07}).

\item By Lemma \ref{lem-w} and a simple computation, we get the estimate from Lemma \ref{lem-acr-07}.

\item By the definition, for all $x\in\Omega_k$, $w_{Q_j}(x)\leq w_{Q_k}(x)$ for every $1\leq j\leq K$. Hence by Lemma~\ref{lem-w} and Lemma~\ref{lem-acr-07}, we have
\begin{equation*}
\int_{\Omega_k}w_{Q_j}^pw_{Q_k}\,dx
\leq\int_{\Omega_k}w_{Q_j}^{\frac{p+1}{2}}w_{Q_k}^{\frac{p+1}{2}}\,dx\leq Ce^{-\frac{p+1}{2}|Q_j-Q_k|}|Q_j-Q_k|^{-\frac{N-3}{2}},
\end{equation*}
and
\begin{align*}
\int_{\Omega_k}w_{Q_j}^2w_{Q_k}^{p-1}\,dx
&\leq C\int_{\Omega_k}w_{Q_j}^{\min\{2,\,\frac{p+1}{2}\}}w_{Q_k}^{\min\{2,\,\frac{p+1}{2}\}}\,dx\\
&\leq Ce^{-\min\{2,\,\frac{p+1}{2}\}|Q_j-Q_k|}|Q_j-Q_k|^{-\frac{N-3}{2}}.
\end{align*}
\end{enumerate}
\end{proof}

Using above integral estimates, we can get the expansion of the energy of approximate solution.
\begin{lemma}\label{lem-energy}
For $K$ sufficiently large, for any $\alpha\in\mathbb{R}$ and $\textbf{q}$ satisfies \eqref{cond-q} we have
\begin{align*}
\mathcal{E}(U)&=KI_0+\big(a_0+o(1)\big)\sum_{j=1}^K|Q_j|^{-m}
-\frac{1}{2}\sum_{i\neq j}\big(\gamma_0+o(1)\big)w(|Q_i-Q_j|)\\
&\quad+O(KR^{-2m})+O\big(Ke^{-\min\{2,\frac{p+1}{2}\}d}d^{-\frac{N-3}{2}}\big),
\end{align*}
where $\gamma_0=\int_{\mathbb{R}^N}w^p(x)e^{-x_1}\,dx$ is a positive constant given in Lemma~\ref{lem-act},
\begin{align*}
I_0=\big(\frac{1}{2}-\frac{1}{p+1}\big)\int_{\mathbb{R}^N}w^{p+1}\,dx,\
\text{and}\
a_0=\frac{a}{2}\int_{\mathbb{R}^N}w^2\,dx.
\end{align*}
\end{lemma}
\begin{proof}
The proof is delayed to Appendix B.
\end{proof}


\section{The Lyapunov-Schmidt reduction}

The aim of this section is to achieve Step 1 in the procedure of our construction described in Section 2.

Before stating the main result, we first introduce some notation. Let $\eta\in(0,1)$ be a constant chosen later, we define the weighted norm:
\begin{equation}\label{norm**}
\|h\|_{**}=\sup_{x\in\mathbb{R}^N}\Big|\big(\sum_{j=1}^Ke^{-\eta|x-Q_j|}\big)^{-1}\,h(x)\Big|,
\end{equation}
where $Q_j$ is defined in \eqref{eq-Q}. In  what follows, we assume that $(Q_1,\dots,Q_K)\in\Lambda_K$, i.e., the parameter $\textbf{q}$ satisfies \eqref{cond-q}.

We first claim that
\begin{equation}\label{norms-**-q}
\|h\|_{L^\infty(\mathbb{R}^N)}
\leq C\|h\|_{**}\ \text{and}\
\|h\|_{L^q(\mathbb{R}^N)}
\leq CK\|h\|_{**}\ \text{for}\ 1\leq q<\infty.
\end{equation}
Indeed, the second inequality in \eqref{norms-**-q} follows directly from
\begin{equation*}
|h(x)|\leq\|h\|_{**}\sum_{j=1}^Ke^{-\eta|x-Q_j|},\ \forall x\in\mathbb{R}^N.
\end{equation*}
To prove the first inequality in \eqref{norms-**-q}, it suffices to show that $\sum_{j=1}^Ke^{-\eta|x-Q_j|}\leq C$. Indeed, for any $x\in\mathbb{R}^N$, we can choose $\ell\in\mathbb{N}_+$ and $j_0\in\{1,\dots,K\}$ such that $x\in\Omega_{j_0}^\ell\setminus\Omega_{j_0}^{\ell-1}$. Hence by Lemma \ref{lem-3.5},
\begin{equation}\label{eq-W-b}
0<\sum_{j=1}^Ke^{-\eta|x-Q_j|}\leq C\ell^{N-1}e^{-\eta(\ell-1)\rho/2}\leq C.
\end{equation}

Denote $\mathcal{B}_{**}=\big\{h\in L^\infty(\mathbb{R}^N)\,\big|\,\|h\|_{**}<\infty\big\}$. Then $\mathcal{B}_{**}$ is a Banach space with the norm $\|h\|_{**}$. To show the completeness, suppose that $\{h_n\}$ is a Cauchy sequence in $\mathcal{B}_{**}$. By \eqref{norms-**-q}, $\{h_n\}$ is also a Cauchy sequence in $L^\infty(\mathbb{R}^N)$. Hence $h_n$ converges to a function $h_\infty$ in $L^\infty(\mathbb{R}^N)$. By the definition of Cauchy sequence, for any $\varepsilon>0$, there is $n_0\in\mathbb{N}$ such that
\begin{equation*}
|h_n(x)-h_k(x)|\Big(\sum_{j=1}^Ke^{-\eta|x-Q_j|}\Big)^{-1}\leq\|h_n-h_k\|_{**}
<\varepsilon,\
\forall\,x\in\mathbb{R}^N,\ \text{if}\ n,k\geq n_0.
\end{equation*}
Letting $k\rightarrow\infty$, we get
\begin{equation*}
|h_n(x)-h_\infty(x)|\Big(\sum_{j=1}^Ke^{-\eta|x-Q_j|}\Big)^{-1}
<\varepsilon,\
\forall\,x\in\mathbb{R}^N,\ \text{if}\ n\geq n_0,
\end{equation*}
which implies that $\|h_n-h_\infty\|_{**}\rightarrow0$ as $n\rightarrow\infty$.

Now we can state our main result in this section.

\begin{proposition}\label{prop-3-1}
Suppose that $V(x)$ satisfies \eqref{cond-wy} for constants $V_\infty>0$, $a\in\mathbb{R}$, $m>0$ and $\sigma>0$. If $N\geq3$, we further assume \eqref{cond-even-1}. Then there is a positive integer $K_0$ such that: for all $K\geq K_0$, every $\alpha\in\mathbb{R}$, and $\textbf{q}$ satisfies \eqref{cond-q}, there exists a unique function $\phi\in W^{2,2}(\mathbb{R}^N)\cap\mathcal{B}_K$ and a unique multiplier $\widehat{\beta}\in\mathbb{R}^{2K}$ such that
\begin{equation}\label{eq-u-1}
\begin{cases}
L[\phi]+E+N(\phi)=\widehat{\beta}\cdot\frac{\partial U}{\partial\textbf{q}},
\vspace{1mm}\\
\int_{\mathbb{R}^N}\phi\,\mathcal{Z}_{Q_j}\,dx=0,\ \forall\,j=1,\dots,K,
\end{cases}
\end{equation}
where
\begin{align*}
\mathcal{B}_K=\left\{\phi\in L^\infty(\mathbb{R}^N)\,:\,\|\phi\|_{**}\leq C_0K^{-\min\{1,\frac{p-\eta}{2}\}m}(\ln K)^{-\frac{1}{2}}\right\}.
\end{align*}
Here $C_0$ is a positive constant independent of $K$. Moreover, $(\alpha,\textbf{q})\mapsto\phi(x;\alpha,\textbf{q})$ is of class $C^1$, and
\begin{equation*}
R^{-1}\|\frac{\partial \phi}{\partial\alpha}\|_{**}+\|\frac{\partial \phi}{\partial\textbf{q}}\|_{**}
\leq C\big(K^{-\min\{1,\frac{p-\eta}{2}\}m}(\ln K)^{-\frac{1}{2}}\big)^{\min\{p-1,1\}}.
\end{equation*}

\end{proposition}

The proof of Proposition~\ref{prop-3-1} is somewhat standard and can be divided into two steps:
\begin{enumerate}[(i)]
\item study the invertibility of the linear operator;
\item apply fixed point theorems.
\end{enumerate}

\subsection{Linear analysis}

Let $M$ denotes an $2K\times2K$ matrix defined by
\begin{equation}\label{matrix-M}
M_{jk}=\int_{\mathbb{R}^N}\frac{\partial U}{\partial q_j}\frac{\partial U}{\partial q_k}\,dx,\ \forall\,j,k=1,\dots,2K.
\end{equation}

\begin{lemma}\label{lem-M}
For $K$ sufficiently large, given any vector $\vec{b}\in\mathbb{R}^{2K}$, there exists a unique vector $\widehat{\beta}\in\mathbb{R}^{2K}$ such that $M\widehat{\beta}=\vec{b}$. Moreover,
\begin{equation}\label{eq-M-b}
\|\widehat{\beta}\|_\infty\leq C\|\vec{b}\|_\infty,
\end{equation}
for some constant $C$ independent of $K$.
\end{lemma}

\begin{proof}
To prove the existence, it is sufficient to prove the a priori estimate~\eqref{eq-M-b}. Suppose that $|\widehat{\beta}_j|=\|\widehat{\beta}\|_\infty$, by the definition, we have
\begin{equation}\label{eq-lem-M}
\sum_{k=1}^KM_{jk}\widehat{\beta}_k=b_j.
\end{equation}
For the entries $M_{jk}$, by Lemma~\ref{lem-d-R} and Lemma~\ref{lem-act}, we get
\begin{equation}\label{eq-lem-M-1}
|M_{jk}|\leq Ce^{-d}d^{-\frac{N-3}{2}}\leq CK^{-m}(m\ln K)^{-m},\ \forall\,k\neq j,
\end{equation}
and
\begin{equation}\label{eq-lem-M-2}
M_{jj}=\int_{\mathbb{R}^N}\big(\frac{\partial w}{\partial x_1}\big)^2\,dx=c_0>0,\ \forall\,j=1,\dots,2K.
\end{equation}
Hence by \eqref{eq-lem-M}-\eqref{eq-lem-M-2}, for $K$ sufficiently large, we have
\begin{equation*}
c_0\|\widehat{\beta}\|_\infty
\leq c_0|\widehat{\beta}_j|
\leq\sum_{k\neq j}|M_{jk}||\widehat{\beta}_k|+|b_j|
\leq\frac{c_0}{2}\|\widehat{\beta}\|_\infty+\|\vec{b}\|_\infty,
\end{equation*}
from which the desired result follows.
\end{proof}

We can now formulate our main result in this subsection.

\begin{lemma}\label{lem-3-1}
Under the assumption of Proposition \ref{prop-3-1}, there is a positive integer $K_0$ such that: for all $K\geq K_0$, every $\alpha\in\mathbb{R}$, and $\textbf{q}$ satisfies \eqref{cond-q}, and for all $h\in\mathcal{B}_{**}$, there exists a unique function $\phi\in W^{2,2}(\mathbb{R}^N)\cap\mathcal{B}_{**}$ and a unique multiplier $\widehat{\beta}\in\mathbb{R}^{2K}$ such that
\begin{equation}\label{eq-phi-beta}
\begin{cases}
L[\phi]=h+\widehat{\beta}\cdot\frac{\partial U}{\partial\textbf{q}},
\vspace{1mm}\\
\int_{\mathbb{R}^N}\phi\,\mathcal{Z}_{Q_j}\,dx=0,\ \forall\,j=1,\dots,K.
\end{cases}
\end{equation}
Moreover, we have
\begin{equation}\label{estimate}
\|\phi\|_{**}+\|\widehat{\beta}\|_\infty\leq C\|h\|_{**},
\end{equation}
for some positive constant $C$ independent of $K$.
\end{lemma}

\begin{proof}
To solve \eqref{eq-phi-beta}, we first consider weak solutions. Define
\begin{align*}
\mathcal{H}=\Big\{u\in H^1(\mathbb{R}^N)\,\big|\,\big(u,\,(-\Delta+1)^{-1}\mathcal{Z}_{Q_j}\big)=0,\,\forall\,j=1,\dots,K\Big\}.
\end{align*}
Then $\mathcal{H}$ is a Hilbert space with the standard inner product:
\begin{align*}
(u,v)=\int_{\mathbb{R}^N}(\nabla u\nabla v+uv)\,dx.
\end{align*}
Since the vector function $\mathcal{Z}_{Q_j}$ decays exponentially at infinity, by integration by parts, it is not hard to show that for $\phi\in H^1(\mathbb{R}^N)$, $\phi\in\mathcal{H}$ is equivalent to
\begin{equation*}
\int_{\mathbb{R}^N}\phi\,\mathcal{Z}_{Q_j}\,dx=0,\,\forall\,j=1,\dots,K.
\end{equation*}

As usual, $\phi\in\mathcal{H}$ is a weak solution of \eqref{eq-phi-beta} if and only if it satisfies the following equation:
\begin{align*}
\int_{\mathbb{R}^N}\Big\{\nabla\phi\nabla\varphi+V(x)\phi\varphi-pU^{p-1}\phi\varphi\Big\}\,dx=\int_{\mathbb{R}^N} h\varphi\,dx,\ \forall\,\varphi\in\mathcal{H}.
\end{align*}
By the Riesz representation theorem, the last equation can be written as
\begin{align*}
\phi+\mathcal{K}[\phi]=\widehat{h},
\end{align*}
where $\widehat{h}$ is defined by duality and $\mathcal{K}$ is a linear compact operator due to the exponential decay of $U$ and $|V(x)-1|\leq C|x|^{-m}$ for $|x|$ large. Using the Fredholm alternative, showing that equation \eqref{eq-phi-beta} has a unique weak solution is equivalent to showing that it has a unique solution for $h=0$. Moreover, by \eqref{norms-**-q}, $h\in L^q(\mathbb{R}^N)$ for all $1<q<\infty$. By the standard elliptic regularity results, $\phi\in W^{2,q}(\mathbb{R}^N)$. Hence $\phi$ is a strong solution and $\phi\in L^\infty(\mathbb{R}^N)$ by the Sobolev imbedding theorem. Therefore, to prove Lemma~\ref{lem-3-1}, it is sufficient to prove the a priori estimate \eqref{estimate}.

To prove \eqref{estimate}, we first multiply equation \eqref{eq-phi-beta} by $\frac{\partial U}{\partial\textbf{q}}$ and integrate over $\mathbb{R}^N$ to obtain
\begin{align}\label{eq-beta}
M\widehat{\beta}
=\int_{\mathbb{R}^N}L[\phi]\,\frac{\partial U}{\partial\textbf{q}}\,dx-\int_{\mathbb{R}^N}h\,\frac{\partial U}{\partial\textbf{q}}\,dx,
\end{align}
where $M$ is an $2K\times2K$ matrix defined in \eqref{matrix-M}.

By the integration by parts,
\begin{align*}
\int_{\mathbb{R}^N}L[\phi]\,\mathcal{Z}_{Q_k}\,dx
=\int_{\mathbb{R}^N}\phi\,L[\mathcal{Z}_{Q_k}]\,dx.
\end{align*}
Observe that
\begin{align*}
L[\mathcal{Z}_{Q_k}]
&=\left(V(x)-1\right)\nabla w_{Q_k}-p\big(U^{p-1}-w_{Q_k}^{p-1}\big)\nabla w_{Q_k}.
\end{align*}
We claim that
\begin{align}\label{LZ}
\Big|\int_{\mathbb{R}^N}L[\phi]\,\mathcal{Z}_{Q_k}\,dx\Big|
\leq Cde^{-\min\{1,\frac{p}{2}\}d}\|\phi\|_\infty.
\end{align}
Indeed, on one hand, by the assumption \eqref{cond-wy} and $\phi\in\mathcal{H}$, we have
\begin{align*}
\Big|\int_{\mathbb{R}^N}\left(V(x)-1\right)\nabla w_{Q_k}\phi\,dx\Big|
\leq C(R^{-m-1}\ln K+R^{-m-\sigma})\|\phi\|_\infty.
\end{align*}
On the other hand, by mean value theorem and \eqref{ineq-q}, for $|x-Q_k|<2m\ln K$, we have
\begin{align*}
|U^{p-1}-w_{Q_k}^{p-1}|\leq Cw_{Q_k}^{p-2}\sum_{j\neq k}w_{Q_j}.
\end{align*}
Thus by Lemma~\ref{lem-act},
\begin{align*}
\Big|\int_{\mathbb{R}^N}-p(U^{p-1}-w_{Q_k}^{p-1})\nabla w_{Q_k}\,\phi\,dx\Big|
\leq Cde^{-\min\{1,\frac{p}{2}\}d}\|\phi\|_\infty.
\end{align*}
Combining the above estimates we get \eqref{LZ}.

Since $w$ decays exponentially at infinity, we have
\begin{align}\label{hZ}
\Big|\int_{\mathbb{R}^N}h\mathcal{Z}_{Q_k}\,dx\Big|\leq C\|h\|_{**}.
\end{align}
Combining the above estimates \eqref{LZ} and \eqref{hZ}, by Lemma~\ref{lem-M}, we get
\begin{align}\label{beta}
\|\widehat{\beta}\|_\infty\leq C\left(de^{-\min\{1,\frac{p}{2}\}d}\|\phi\|_\infty+\|h\|_{**}\right).
\end{align}

Now we prove the a priori estimate~\eqref{estimate}. First we show that $\|\phi\|_{**}<\infty$. To prove it, by the maximum principle, we prove that there exist constants $\tau$ and $C$ (all independent of $K$) such that for all $x\in\mathbb{R}^N\setminus\cup_{j=1}^K\,B(Q_j,\tau)$,
\begin{align}\label{pe}
|\phi(x)|\leq C\left(\|L[\phi]\|_{**}+\sup_{1\leq j\leq K}\|\phi\|_{L^\infty(B(Q_j,\tau))}\right)\sum_{j=1}^Ke^{-\eta|x-Q_j|}.
\end{align}

To prove the above pointwise estimate, we first show the independence of $\tau$ on $K$, for $x\in\mathbb{R}^N\setminus\cup_{j=1}^K\,B(Q_j,\tau)$, by Lemma~\ref{lem-N}, we have
\begin{align*}
U(x)
&\leq\sum_{|Q_j-x|<\rho/2}w(x-Q_j)+\sum_{\ell=1}^\infty\sum_{\ell\rho/2\leq|Q_j-x|<(\ell+1)\rho/2}w(x-Q_j)\nonumber\\
&\leq w(\tau)+C\sum_{\ell=1}^\infty \ell^{N-1}e^{-\ell\rho/2}\leq Cw(\tau).
\end{align*}
Thus we can take $\tau$ sufficiently large but independent of $K$ such that
\begin{equation}\label{constant-tau}
pU^{p-1}(x)\leq(V_0-\eta^2)/4,\ \forall\,x\in\mathbb{R}^N\setminus\cup_{j=1}^K\,B(Q_j,\tau).
\end{equation}

Now we claim that for $\tau$ sufficiently large (independent of $K$), in $\mathbb{R}^N\setminus\cup_{j=1}^K\,B(Q_j,\tau)$,
\begin{align*}
L[W_{-}]\geq c_0W_{-},\ \text{and}\ L[W_+]\geq c_0W_+
\end{align*}
where $W_{\pm}(x)=\sum_{j=1}^Ke^{\pm\eta|x-Q_j|}$ and $c_0>0$ is a constant independent of $K$. Indeed, for $x\in\mathbb{R}^N\setminus\cup_{j=1}^K\,B(Q_j,\tau)$,
\begin{align*}
L[W_{\pm}]=\sum_{j=1}^K\Big\{V(x)-\eta^2\mp\frac{N-1}{|x-Q_j|}\eta-pU^{p-1}\Big\}e^{\pm\eta|x-Q_j|}\geq\frac{V_0-\eta^2}{2}W_\pm,
\end{align*}
by the assumption $(V1)$ and inequality~\eqref{constant-tau}.

The remaining part  in the proof of \eqref{pe} is to apply the maximum principle for the linear operator $L$ in $\mathbb{R}^N\setminus\cup_{j=1}^K\,B(Q_j,\tau)$ to obtain
\begin{align*}
|\phi(x)|\leq C\left(\|L[\phi]\|_{**}+\sup_{1\leq j\leq K}\|\phi\|_{L^\infty(B(Q_j,\tau))}\right)\sum_{j=1}^Ke^{-\eta|x-Q_j|}+\delta\sum_{j=1}^Ke^{\eta|x-Q_j|}
\end{align*}
for any $\delta>0$, where $C$ is a constant independent of $K$ and $\delta$. Letting $\delta\rightarrow0$, we get the desired estimate \eqref{pe}. Hence
\begin{equation}\label{pe-2}
\|\phi\|_{**}
\leq C\left(\|L[\phi]\|_{**}+\sup_{1\leq j\leq K}\|\phi\|_{L^\infty(B(Q_j,\tau))}\right)<\infty.
\end{equation}

Now we can prove the a priori estimate \eqref{estimate}. Arguing by contradiction, assume that there is a sequence of $(\phi^{(K)},h^{(K)})$ satisfying \eqref{eq-phi-beta} such that
\begin{align*}
\|\phi^{(K)}\|_{**}=1,\quad\text{and}\quad \|h^{(K)}\|_{**}=o(1),\ \text{as}\ K\rightarrow\infty.
\end{align*}
(For simplicity, in the following we will drop $(K)$ in the superscript) As a consequence of \eqref{beta},
\begin{align*}
|\widehat{\beta}\cdot\frac{\partial U}{\partial\textbf{q}}(x)|
&\leq C\left(de^{-\min\{1,\frac{p}{2}\}d}\|\phi\|_\infty+\|h\|_{**}\right)\sum_{j=1}^Ke^{-\eta|x-Q_j|}.
\end{align*}
Since $\|\phi\|_\infty\leq C\|\phi\|_{**}$ and $\|h\|_{**}=o(1)$, we get $\|L[\phi]\|_{**}=o(1)$. Hence \eqref{pe-2} implies that there exists a subsequence of $Q_j$ such that
\begin{align}\label{Q_j}
\|\phi\|_{L^\infty(B(Q_{j},\tau))}\geq C>0
\end{align}
for some fixed constant $C$ (independent of $K$). Since $\|\phi\|_\infty\leq 1$, by elliptic regularity estimates, we get $\|\phi\|_{C^{1}(\mathbb{R}^N)}\leq C$. Applying Ascoli-Arzela's theorem, one can find a subsequence of $Q_j$ such that $\phi(x+Q_j)$ converge (on compact sets) to $\phi_\infty$. It is not hard to show that $\phi_\infty$ is a bounded (weak and then strong) solution (actually bound by $e^{-\eta|x|}$) of
\begin{equation*}
-\Delta\phi_\infty+\phi_\infty-pw^{p-1}\phi_\infty=0.
\end{equation*}
Furthermore, since $\phi$ satisfies the orthogonality condition $\int_{\mathbb{R}^N}\phi\,\mathcal{Z}_{Q_j}\,dx=0$, the limit function $\phi_\infty$ satisfies $\int_{\mathbb{R}^N}\phi_\infty\nabla w=0$. By the non-degeneracy of $w$, one has $\phi_\infty\equiv0$, which is in contradiction with \eqref{Q_j}. This completes the proof of Lemma \ref{lem-3-1}.
\end{proof}

\begin{remark}
If $V(x)$ is a bounded measurable function such that there is no nontrivial solution of
\begin{equation}
-\Delta\phi+V(x)\phi=0, \ |\phi(x)|\leq Ce^{-\eta|x|}\ \text{in}\ \mathbb{R}^N,
\end{equation}
our arguments still work by adding $0$ to the points $Q_j$'s.
\end{remark}

\begin{remark}
Since the Morse index of $w$ is finite, using a similar argument in the proof of Lemma~\ref{lem-3-1} (cf. \cite{AW-12}), one can show that
\begin{equation}\label{H1}
\|\phi\|_{H^1(\mathbb{R}^N)}\leq C\|h\|_{L^2(\mathbb{R}^N)}
\end{equation}
for some positive constant $C$ independent of $K$. Indeed, since $\tau$ is independent of $K$,
one can first prove that
\begin{equation*}
C\|\phi\|_{H^1(\mathbb{R}^N)}^2
\leq\int_{\mathbb{R}^N}\Big\{|\nabla\phi|^2+V(x)\phi^2-pU^{p-1}\phi^2\Big\}\,dx.
\end{equation*}
\end{remark}

\subsection{Nonlinear analysis}

Summarizing, for any $h\in\mathcal{B}_{**}$, by Lemma \ref{lem-3-1}, there is a unique function $\phi\in\mathcal{H}\cap W^{2,2}(\mathbb{R}^N)\cap\mathcal{B}_{**}$ satisfying \eqref{eq-phi-beta}. Hence we can define a linear operator from $\mathcal{B}_{**}$ to $\mathcal{H}\cap W^{2,2}(\mathbb{R}^N)\cap\mathcal{B}_{**}$ and denote it by $L^{-1}$. Then the equation \eqref{eq-u-1} is equivalent to
\begin{align*}
\phi=-L^{-1}\big[E+N(\phi)\big].
\end{align*}


Before we give the complete proof of Proposition \ref{prop-3-1}, we first show the estimate of the error.

\begin{lemma}\label{lem-error}
Given $(Q_1,\dots,Q_K)\in\Lambda_K$, then for any fixed $0<\eta<1$ and $K$ sufficiently large, there is a constant $C$ (independent of $K$) such that
\begin{equation}\label{eq-error}
\|E\|_{**}\leq CK^{-\min\{1,\frac{p-\eta}{2}\}m}(\ln K)^{-\frac{1}{2}}.
\end{equation}
\end{lemma}

\begin{proof}
By the definition, we have
\begin{align*}
E
&=\underbrace{\sum_{j=1}^K\left(V(x)-1\right)w_{Q_j}}_{E_1}-\underbrace{\Big\{\Big(\sum_{j=1}^Kw_{Q_j}\Big)^p-\sum_{j=1}^Kw_{Q_j}^p\Big\}}_{E_2}.
\end{align*}

{\bf Claim 1:} There exists a constant $C$ (independent of $K$) such that
\begin{equation}\label{E1}
\|E_1\|_{**}\leq CR^{-m}\leq CK^{-m}(\ln K)^{-m}.
\end{equation}

{\bf{Claim 2:}} There exists a constant $C$ (independent of $K$) such that
\begin{align}\label{E2}
\|E_2\|_{**}&\leq Cd^{-\frac{N-1}{2}}e^{-\min\{1,\frac{p-\eta}{2}\}d}
\leq CK^{-\min\{1,\frac{p-\eta}{2}\}m}(\ln K)^{-\min\{\frac{N-1}{2},m+1\}}.
\end{align}

If both Claim 1 and Claim 2 are true, the desired estimate~\eqref{eq-error} follows.

{\bf Proof of Claim 1:} Note that for $|x|<R/3$, by the triangle inequality, we have
\begin{align*}
|x-Q_j|\geq|Q_j|-|x|\geq R/2.
\end{align*}
Hence for all $|x|<R/3$, by $V\in L^\infty(\mathbb{R}^N)$ and Lemma~\ref{lem-w}, we get
\begin{align*}
|E_1(x)|\leq C\sum_{j=1}^Kw(x-Q_j)
\leq Ce^{-(1-\eta)R/2}\sum_{j=1}^K e^{-\eta|x-Q_j|}
\leq CK^{-m-3}\sum_{j=1}^K e^{-\eta|x-Q_j|}.
\end{align*}

For $|x|\geq R/3$, by the assumption~\eqref{cond-wy}, we have $|V(x)-1|\leq CR^{-m}$. Hence for all $|x|\geq R/3$,
\begin{align*}
|E_1(x)|\leq CR^{-m}\sum_{j=1}^Kw(x-Q_j)
\leq CR^{-m}\sum_{j=1}^K e^{-\eta|x-Q_j|}.
\end{align*}
Combining these estimates, Claim 1 follows.

{\bf{Proof of Claim 2:}}  For $x\in\Omega_{K+1}^\ell$, where $\ell\in\mathbb{N}_+$ is chosen later, we have
\begin{align*}
|E_2(x)|
&\leq K^{p-1}\sum_{j=1}^Kw_{Q_j}^p(x)+\sum_{j=1}^Kw_{Q_j}^p(x)\\
&\leq CK^{p-1}\sum_{j=1}^Ke^{-p|x-Q_j|}
\leq CK^{p-1}e^{-(p-\eta)\ell \rho/2}\sum_{j=1}^Ke^{-\eta|x-Q_j|}.
\end{align*}
Since $\rho>\frac{m}{2}\ln K$, by choosing $\ell>\frac{4(p+m+2)}{m(p-1)}$ (independent of $K$), we have
\begin{align*}
K^{p-1}e^{-(p-\eta)\ell\rho/2}\leq K^{p-1}K^{-(p-\eta)\ell m/4}\leq CK^{-m-3}.
\end{align*}
Thus for all $x\in\Omega_{K+1}^\ell$,
\begin{align*}
|E_2(x)|\leq CK^{-m-3}\sum_{j=1}^Ke^{-\eta|x-Q_j|}.
\end{align*}

By the definition of $\Omega_j^\ell$, $j=1,\dots,K$, $|x-Q_j|\leq|x-Q_k|$ for all $x\in\Omega_j^\ell$ and $1\leq k\leq K$. Thus $|x-Q_k|\geq \frac{\rho}{2}$ for $k\neq j$. Hence
by mean value theorem and \eqref{ineq-q}, for all $x\in\Omega_j^\ell$, we have
\begin{align*}
|E_2(x)|\leq\Big|\Big(\sum_{k=1}^Kw_{Q_k}\Big)^p-w_{Q_j}^p\Big|+\sum_{i\neq j}w_{Q_i}^p
&\leq p\Big(\sum_{k=1}^Kw_{Q_k}\Big)^{p-1}\sum_{k\neq j}w_{Q_k}+\sum_{k\neq j}w_{Q_k}^p\\
&\leq C\ell^{(N-1)(p-1)}w_{Q_j}^{p-1}\sum_{k\neq j}w_{Q_k}.
\end{align*}
Since $\ell$ is independent of $K$, for all $x\in\Omega_j^\ell$, by theorem~\ref{lem-w} and \eqref{q-q-K} we have
\begin{align*}
|E_2(x)|
&\leq Ce^{-(p-1)|x-Q_j|}\sum_{k\neq j}\rho^{-\frac{N-1}{2}}e^{-|x-Q_k|}\\
&\leq C\rho^{-\frac{N-1}{2}}e^{-\eta|x-Q_j|}\sum_{k\neq j}e^{-\min\{1,\frac{p-\eta}{2}\}|Q_j-Q_k|}\\
&\leq Cd^{-\frac{N-1}{2}}e^{-\min\{1,\frac{p-\eta}{2}\}d}e^{-\eta|x-Q_j|}.
\end{align*}
Combining these estimates, Claim 2 follows.
\end{proof}

Now we are in the position to give the proof of Proposition \ref{prop-3-1}.

\begin{proof}[Proof of Proposition \ref{prop-3-1}]
Let $C_0$ be a positive number to be determined later, we define
\begin{align*}
\mathcal{B}_K=\left\{\phi\in L^\infty(\mathbb{R}^N)\,:\,\|\phi\|_{**}\leq C_0K^{-\min\{1,\frac{p-\eta}{2}\}m}(\ln K)^{-\frac{1}{2}}\right\}.
\end{align*}
Then $\mathcal{B}_K$ is a non-empty closed set in $\mathcal{B}_{**}$.  Now we define a map $\mathcal{A}:\mathcal{B}_K\mapsto\mathcal{H}\cap W^{2,2}(\mathbb{R}^N)\cap\mathcal{B}_{**}$ by
\begin{align*}
\mathcal{A}(\phi)=-L^{-1}\big[E+N(\phi)\big].
\end{align*}
Now solving equation \eqref{eq-u-1} is equivalent to finding a fixed point for the map $\mathcal{A}$.

Since $\phi$ is uniformly bounded for $\phi\in\mathcal{B}_K$, by the mean value theorem, there is a positive constant $C$ such that for all $\phi\in\mathcal{B}_K$,
\begin{align*}
|N(\phi)|\leq C|\phi|^{\min\{p,2\}},
\end{align*}
and for all $\phi_1,\phi_2\in\mathcal{B}_K$,
\begin{align*}
|N(\phi_1)-N(\phi_2)|\leq C(|\phi_1|^{\min\{p-1,1\}}+|\phi_2|^{\min\{p-1,1\}})|\phi_1-\phi_2|.
\end{align*}
Thus by \eqref{eq-W-b}, one has
\begin{align*}
\|N(\phi)\|_{**}\leq C\|\phi\|_{**}^{\min\{p,2\}},
\end{align*}
and
we have that
\begin{align*}
\|N(\phi_1)-N(\phi_2)\|_{**}\leq C(\|\phi_1\|_{**}^{\min\{p-1,1\}}+\|\phi_2\|_{**}^{\min\{p-1,1\}})\|\phi_1-\phi_2\|_{**}.
\end{align*}
Hence by Lemma \ref{lem-3-1} and Lemma \ref{lem-error}, for $K$ sufficiently large and $C_0$ large we have
\begin{align*}
\|\mathcal{A}(\phi)\|_{**}\leq C(\|E\|_{**}+\|N(\phi)\|_{**})\leq C_0K^{-\min\{1,\frac{p-\eta}{2}\}m}(\ln K)^{-\frac{1}{2}},
\end{align*}
and
\begin{align*}
\|\mathcal{A}(\phi_1)-\mathcal{A}(\phi_2)\|_{**}\leq C\|N(\phi_1)-N(\phi_2)\|_{**}\leq\frac{1}{2}\|\phi_1-\phi_2\|_{**},
\end{align*}
which shows that $\mathcal{A}$ is a contraction mapping on $\mathcal{B}_K$. Hence there is a unique $\phi\in\mathcal{B}_K$ such that \eqref{eq-u-1} holds.

Now we come to the differentiability of $\phi(x;\alpha,\textbf{q})$ of $(\alpha,\textbf{q})$. Consider the following map $\mathcal{T}:\mathbb{R}\times\mathbb{R}^{2K}\times\mathcal{B}\times\mathbb{R}^{2K}\rightarrow\mathcal{B}\times\mathbb{R}^{2K}$ of class $C^1$:
\begin{align*}
\mathcal{T}(\alpha,\textbf{q},\phi,\widehat{\beta})=\begin{pmatrix}
(-\Delta+1)^{-1}S(U+\phi)-\widehat{\beta}\cdot (-\Delta+1)^{-1}\frac{\partial U}{\partial\textbf{q}}
\vspace{1mm}\\
\int_{\mathbb{R}^N}\phi\,\mathcal{Z}_{Q_1}\,dx\\
\vdots\\
\int_{\mathbb{R}^N}\phi\,\mathcal{Z}_{Q_K}\,dx
\end{pmatrix},
\end{align*}
where $\mathcal{B}=W^{2,2}(\mathbb{R}^N)\cap\mathcal{B}_{**}$.
Equation \eqref{eq-u-1} is equivalent to $\mathcal{T}(\alpha,\textbf{q},\phi,\widehat{\beta})=0$. By the above argument, we know that, given $\alpha\in\mathbb{R}$ and $\textbf{q}$ satisfying \eqref{cond-q}, there is a unique local solution $(\phi(\alpha,\textbf{q}),\widehat{\beta}(\alpha,\textbf{q}))$. For simplicity, in the following, we write $(\phi,\widehat{\beta})=(\phi(\alpha,\textbf{q}),\widehat{\beta}(\alpha,\textbf{q}))$. We claim that the linear operator
\begin{align*}
\frac{\partial\,\mathcal{T}(\alpha,\textbf{q},\phi,\widehat{\beta})}{\partial(\phi,\widehat{\beta})}\Big|_{(\alpha,\textbf{q},\phi,\widehat{\beta})}:\mathcal{B}\times\mathbb{R}^{2K}\rightarrow\mathcal{B}\times\mathbb{R}^{2K}
\end{align*}
is invertible for $K$ large. Then the $C^1$-regularity of $(\alpha,\textbf{q})\mapsto(\phi,\widehat{\beta})$ follows from the Implicit Function Theorem. Indeed we have
\begin{align*}
&\frac{\partial\,\mathcal{T}(\alpha,\textbf{q},\phi,\widehat{\beta})}{\partial(\phi,\widehat{\beta})}\Big|_{(\alpha,\textbf{q},\phi,\widehat{\beta})}[\varphi,\vec{\zeta}]\\
&=\begin{pmatrix}
(-\Delta+1)^{-1}S'(U+\phi)[\varphi]-\vec{\zeta}\cdot (-\Delta+1)^{-1}\frac{\partial U}{\partial\textbf{q}}
\vspace{1mm}\\
\int_{\mathbb{R}^N}\varphi\,\mathcal{Z}_{Q_1}\,dx\\
\vdots\\
\int_{\mathbb{R}^N}\varphi\,\mathcal{Z}_{Q_K}\,dx
\end{pmatrix}.
\end{align*}
Since $\|\phi\|_{**}\leq C_0K^{-\min\{1,\frac{p-\eta}{2}\}m}(\ln K)^{-\frac{1}{2}}$, by Lemma~\ref{lem-M}, the argument in the proof of Lemma \ref{lem-3-1} shows that $\frac{\partial\,\mathcal{T}(\alpha,\textbf{q},\phi,\widehat{\beta})}{\partial(\phi,\widehat{\beta})}\Big|_{(\alpha,\textbf{q},\phi,\widehat{\beta})}
$ is invertible for $K$ sufficiently large. This concludes the proof of Proposition \ref{prop-3-1}.

Next we study the dependence of $\phi$ on $(\alpha,\textbf{q})$. Assume that we have two solutions corresponding to two sets of parameters. One of them denoted by
\begin{align*}
L[\phi]+E+N(\phi)=\widehat{\beta}\cdot\nabla_\textbf{q}U,
\end{align*}
corresponds to the parameters $\alpha$ and $\textbf{q}$; the other denoted by
\begin{align*}
\mathring{L}[\mathring{\phi}]+\mathring{E}+\mathring{N}(\mathring{\phi})=\mathring{\widehat{\beta}}\cdot\mathring{\nabla}_\textbf{q}U,
\end{align*}
corresponds to the parameters $\mathring{\alpha}$ and $\mathring{\textbf{q}}$. Observe that $\phi$ is $L^2$-orthogonal to $\nabla_\textbf{q}U$ while $\mathring{\phi}$ is $L^2$-orthogonal to $\mathring{\nabla}_\textbf{q}U$. To compare $\mathring{\phi}$ with $\phi$, we first choose a vector $\vec{\omega}$ so that
\begin{align*}
\mathring{\phi}_\omega=\mathring{\phi}+\vec{\omega}\cdot\nabla_\textbf{q}U
\end{align*}
satisfies the same orthogonality condition as $\phi$. Moreover, by the equation of $\mathring{\phi}$, the function $\mathring{\phi}_\omega$ satisfies the equation
\begin{align*}
L[\mathring{\phi}_\omega]+(\mathring{L}-L)[\mathring{\phi}]-\vec{\omega}\cdot L[\nabla_\textbf{q}U]+\mathring{E}+\mathring{N}(\mathring{\phi})+\mathring{\widehat{\beta}}\cdot(\nabla_\textbf{q}U-\mathring{\nabla}_\textbf{q}U)=\mathring{\widehat{\beta}}\cdot\nabla_\textbf{q}U.
\end{align*}
Taking the difference with the equation satisfied by $\phi$, we get
\begin{align*}
L[\mathring{\phi}_\omega-\phi]
&=(L-\mathring{L})[\mathring{\phi}]+\vec{\omega}\cdot L[\nabla_\textbf{q}U]
+(E-\mathring{E})+(N(\phi)-\mathring{N}(\mathring{\phi}))\\
&\qquad-\mathring{\widehat{\beta}}\cdot(\nabla_\textbf{q}U-\mathring{\nabla}_\textbf{q}U)
+(\mathring{\widehat{\beta}}-\widehat{\beta})\cdot\nabla_\textbf{q}U.
\end{align*}
Note that by \eqref{eq-Q}, for $j=1,\dots,K$, we have
\begin{align*}
|\mathring{Q}_j-Q_j|\leq C(R|\mathring{\alpha}-\alpha|+\|\mathring{\textbf{q}}-\textbf{q}\|_\infty).
\end{align*}
Assume that $(R|\mathring{\alpha}-\alpha|+\|\mathring{\textbf{q}}-\textbf{q}\|_\infty)\leq1/2$, then we have
\begin{align*}
\|(L-\mathring{L})[\mathring{\phi}]\|_{**}
&\leq CK^{-\min\{1,\frac{p-\eta}{2}\}m}(\ln K)^{-\frac{1}{2}}(R|\mathring{\alpha}-\alpha|+\|\mathring{\textbf{q}}-\textbf{q}\|_\infty),
\end{align*}
\begin{align*}
\|\vec{\omega}\cdot L[\nabla_\textbf{q}U]\|_{**}
\leq CK^{-\min\{1,\frac{p-\eta}{2}\}m}(\ln K)^{-\frac{1}{2}}\|\vec{\omega}\|_\infty,
\end{align*}
\begin{align*}
\|E-\mathring{E}\|_{**}\leq CK^{-\min\{1,\frac{p-\eta}{2}\}m}(\ln K)^{-\frac{1}{2}}(R|\mathring{\alpha}-\alpha|+\|\mathring{\textbf{q}}-\textbf{q}\|_\infty),
\end{align*}
\begin{align*}
\|N(\phi)-\mathring{N}(\mathring{\phi})\|_{**}
&\leq C(\|\phi\|_{**}^{p-1}+\|\mathring{\phi}\|_{**}^{p-1})\|\phi-\mathring{\phi}\|_{**}
+C\|\phi\|_{**}^{\min\{p-1,1\}}(R|\mathring{\alpha}-\alpha|+\|\mathring{\textbf{q}}-\textbf{q}\|_\infty)\\
&\leq CK^{-\min\{1,\frac{p-\eta}{2}\}m(p-1)}(\ln K)^{-\frac{p-1}{2}}\|\phi-\mathring{\phi}\|_{**}\\
&\quad+C\big(K^{-\min\{1,\frac{p-\eta}{2}\}m}(\ln K)^{-\frac{1}{2}}\big)^{\min\{p-1,1\}}(R|\mathring{\alpha}-\alpha|+\|\mathring{\textbf{q}}-\textbf{q}\|_\infty),
\end{align*}
\begin{align*}
\|\mathring{\widehat{\beta}}\cdot(\nabla_\textbf{q}U-\mathring{\nabla}_\textbf{q}U)\|_{**}
&\leq C\|\mathring{\widehat{\beta}}\|_\infty(R|\mathring{\alpha}-\alpha|+\|\mathring{\textbf{q}}-\textbf{q}\|_\infty)\\
&\leq CK^{-\min\{1,\frac{p-\eta}{2}\}m}(\ln K)^{-\frac{1}{2}}(R|\mathring{\alpha}-\alpha|+\|\mathring{\textbf{q}}-\textbf{q}\|_\infty).
\end{align*}
Hence by Lemma~\ref{lem-3-1},
\begin{align*}
\|\mathring{\phi}_\omega-\phi\|_{**}
+\|\mathring{\widehat{\beta}}-\widehat{\beta}\|_\infty
&\leq C\big(K^{-\min\{1,\frac{p-\eta}{2}\}m}(\ln K)^{-\frac{1}{2}}\big)^{\min\{p-1,1\}}(R\|\mathring{\alpha}-\alpha\|_\infty+\|\mathring{q}-q\|_\infty)\\
&\quad+CK^{-\min\{1,\frac{p-\eta}{2}\}m}(\ln K)^{-\frac{1}{2}}\|\vec{\omega}\|_\infty\\
&\quad+CK^{-\min\{1,\frac{p-\eta}{2}\}m(p-1)}(\ln K)^{-\frac{p-1}{2}}\|\mathring{\phi}-\phi\|_{**}.
\end{align*}
On the other hand, by the definition of $\mathring{\phi}_\omega$, we have
\begin{align*}
\|\vec{\omega}\|_\infty
&\leq C\|\mathring{\phi}\|_{**}(R\|\mathring{\alpha}-\alpha\|_\infty+\|\mathring{q}-q\|_\infty)\\
&\leq CK^{-\min\{1,\frac{p-\eta}{2}\}m}(\ln K)^{-\frac{1}{2}}(R\|\mathring{\alpha}-\alpha\|_\infty+\|\mathring{q}-q\|_\infty).
\end{align*}
Hence
\begin{align*}
\|\mathring{\phi}-\phi\|_{**}
+\|\mathring{\widehat{\beta}}-\widehat{\beta}\|_\infty
\leq C\big(K^{-\min\{1,\frac{p-\eta}{2}\}m}(\ln K)^{-\frac{1}{2}}\big)^{\min\{p-1,1\}}(R\|\mathring{\alpha}-\alpha\|_\infty+\|\mathring{q}-q\|_\infty).
\end{align*}
Therefore, we conclude that
\begin{equation*}
R^{-1}\|\frac{\partial \phi}{\partial\alpha}\|_{**}+\|\frac{\partial \phi}{\partial\textbf{q}}\|_{**}
\leq C\big(K^{-\min\{1,\frac{p-\eta}{2}\}m}(\ln K)^{-\frac{1}{2}}\big)^{\min\{p-1,1\}}.
\end{equation*}

\end{proof}


\section{A further reduction process}

The main purpose of this section is to achieve Step 2.A. As explained in Section 2, we define
\begin{equation}
\vec{\beta}=\widehat{\beta}-\gamma(R\textbf{q}_0+\textbf{q}^\perp),\ \text{for every }\gamma\in\mathbb{R}.
\end{equation}
Then equation~\eqref{eq-u-1} becomes
\begin{align}\label{phi-4-2}
L[\phi]+E+N(\phi)=\vec{\beta}\cdot\frac{\partial U}{\partial\textbf{q}}+\gamma\frac{\partial U}{\partial \alpha}.
\end{align}
Note that $\phi$ does not depend on $\gamma$, but $\vec{\beta}$ depends on the parameters $\alpha$, $\textbf{q}$ and $\gamma$ and we write $\vec{\beta}=\vec{\beta}(\alpha,\textbf{q},\gamma)$.

In this section we are going to solve $\vec{\beta}(\alpha,\textbf{q},\gamma)=0$ for each $\alpha\in\mathbb{R}$ by adjusting $\gamma$ and $\textbf{q}$. We multiply \eqref{phi-4-2} by $\frac{\partial U}{\partial\textbf{q}}$ and integrate over $\mathbb{R}^N$ to conclude that
\begin{align*}
\int_{\mathbb{R}^N}(E+L[\phi]+N(\phi))\frac{\partial U}{\partial\textbf{q}}\,dx
=M\vec{\beta}+\gamma\int_{\mathbb{R}^N}\frac{\partial U}{\partial \alpha}\frac{\partial U}{\partial\textbf{q}}\,dx.
\end{align*}
By Lemma~\ref{lem-M}, solving $\vec{\beta}(\alpha,\textbf{q},\gamma)=0$ amounts to solve
\begin{align}\label{beta-1}
\int_{\mathbb{R}^N}(E+L[\phi]+N(\phi))\frac{\partial U}{\partial\textbf{q}}\,dx
=\gamma\int_{\mathbb{R}^N}\frac{\partial U}{\partial \alpha}\frac{\partial U}{\partial\textbf{q}}\,dx.
\end{align}
For this purpose, in the next subsection, we will computer the projection of the error and the projections of the terms involving $\phi$.

\subsection{Projections}

We first compute $\int_{\mathbb{R}^N}E\,\frac{\partial U}{\partial\textbf{q}}\,dx$. Recall that
\begin{align*}
\frac{\partial U}{\partial\textbf{q}}
=-(\mathcal{Z}_{Q_1}\cdot\vec{n}_1,\dots,\mathcal{Z}_{Q_K}\cdot\vec{n}_K,\mathcal{Z}_{Q_1}\cdot\vec{t}_1,\dots,\mathcal{Z}_{Q_K}\cdot\vec{t}_K)^T.
\end{align*}

\begin{lemma}\label{lem-4-1}
Under the assumption of Proposition \ref{prop-3-1}, for sufficiently large $K$, the following expansion holds:
\begin{align*}
\int_{\mathbb{R}^N}E\mathcal{Z}_{Q_k}\,dx
&=a_0m|Q_k|^{-m-1}\frac{Q_k}{|Q_k|}+\sum_{j\neq k}\Psi(|Q_j-Q_k|)\frac{(Q_j-Q_k)}{|Q_j-Q_k|}\\
&\quad+R^{-m-\sigma}\varPi_{k,1}(\alpha,\textbf{q})+R^{-m-3}\varPi_{k,2}(\alpha,\textbf{q})+R^{-2m}\varPi_{k,3}(\alpha,\textbf{q})\\
&\quad+R^{-\min\{2-\eta,\frac{p+1-\eta}{2}\}m}\varPi_{k,4}(\alpha,\textbf{q})
\end{align*}
where $\eta$ is a small positive constant chosen later, $a_0=\frac{a}{2}\int_{\mathbb{R}^N}w^2(x)\,dx$, and $\varPi_{k,l}(\alpha,\textbf{q})$'s are smooth vector valued functions, which are uniformly bounded as $K\rightarrow\infty$.
\end{lemma}

\begin{proof}
By the definition,
\begin{align*}
\int_{\mathbb{R}^N}E\mathcal{Z}_{Q_k}\,dx
&=\underbrace{\sum_{j=1}^K\int_{\mathbb{R}^N}(V(x)-1)w_{Q_j}\nabla w_{Q_k}\,dx}_{I_1}
-\underbrace{\int_{\mathbb{R}^N}\Big\{\Big(\sum_{j=1}^Kw_{Q_j}\Big)^p-\sum_{j=1}^Kw_{Q_j}^p\Big\}\nabla w_{Q_k}\,dx}_{I_2}.
\end{align*}

\noindent{\bf{Claim 1:}} There are smooth vector valued functions $\varPi_{k,1}$, $\varPi_{k,2}$ and $\varPi_{k,3}$ of $\alpha$ and $\textbf{q}$ such that
\begin{align*}
I_{1}&=a_0m|Q_k|^{-m-1}\frac{Q_k}{|Q_k|}
+R^{-m-\sigma}\varPi_{k,1}(\alpha,\textbf{q})+R^{-m-3}\varPi_{k,2}(\alpha,\textbf{q})+R^{-2m}\varPi_{k,3}(\alpha,\textbf{q}).
\end{align*}

\noindent{{\bf Claim 2:}} There exists a smooth vector valued function $\varPi_{k,4}$ of $\alpha$ and $\textbf{q}$ such that
\begin{align*}
I_2=-\sum_{j\neq k}\Psi(|Q_j-Q_k|)\frac{(Q_j-Q_k)}{|Q_j-Q_k|}
-R^{-\min\{2-\eta,\frac{p+1-\eta}{2}\}m}\varPi_{k,4}(\alpha,\textbf{q}),
\end{align*}
where $\eta$ is a small positive constant chosen later.

Combining Claim 1 and Claim 2, we get the desired result. The remainder of this proof will be devoted to the proofs of Claim 1 and Claim 2.

\noindent{\bf{Proof of Claim 1:}}
We divide $I_1$ into two parts:
\begin{align*}
I_1
=\underbrace{\int_{\mathbb{R}^N}(V(x)-1)w_{Q_k}\nabla w_{Q_k}\,dx}_{I_{11}}
+\underbrace{\sum_{j\neq k}\int_{\mathbb{R}^N}(V(x)-1)\,w_{Q_j}\nabla w_{Q_k}\,dx}_{I_{12}}.
\end{align*}
Write $I_{11}=I_{111}+I_{112}$, where
\begin{align*}
I_{111}=\int_{|x|\leq\frac{R}{3}}(V(x)-1)w_{Q_k}\nabla w_{Q_k}\,dx,\ \text{and}\ I_{112}=\int_{|x|>\frac{R}{3}}(V(x)-1)w_{Q_k}\nabla w_{Q_k}\,dx.
\end{align*}
On one hand, since $V\in L^\infty(\mathbb{R}^N)$ and \eqref{eq-rho}, by the triangle inequality and the property of $w$, we get $|I_{111}|\leq Ce^{-R/2}\leq CR^{-m-3}$. On the other hand, by \eqref{cond-wy} and Taylor's theorem,
\begin{align*}
I_{112}&=\int_{|x|>\frac{R}{3}}\frac{a}{|x|^m}w_{Q_k}\nabla w_{Q_k}\,dx+O(R^{-m-\sigma})\\
&=\int_{|y+Q_k|>\frac{R}{3}}\frac{a}{|y+Q_k|^m}w(y)\nabla w(y)\,dy+O(R^{-m-\sigma})\\
&=m|Q_k|^{-m-1}\frac{Q_k}{|Q_k|}\frac{a}{2}\int_{\mathbb{R}^N}w^2(x)\,dx+O(R^{-m-\sigma})+O(R^{-m-3}),
\end{align*}
where in the last equality we use the following identities:
\begin{align*}
\int_{\mathbb{R}^N}w(y)\nabla w(y)\,dy=0,\
\int_{\mathbb{R}^N}y_jy_kw(y)\nabla w(y)\,dy=0,\ \forall\,j,k=1,\dots,N,
\end{align*}
and
\begin{align*}
\int_{\mathbb{R}^N}(y^T\vec{e})w(y)\nabla w(y)\,dy=-\frac{1}{2}\int_{\mathbb{R}^N}w^2\,dy\,\vec{e},\ \forall\,\vec{e}\in\mathbb{R}^N.
\end{align*}
Therefore, we get
\[
I_{11}=a_0m|Q_k|^{-m-1}\frac{Q_k}{|Q_k|}++O(R^{-m-\sigma})+O(R^{-m-3}).
\]
If \eqref{cond-wy} holds in the $C^1$ sense, i.e.,
\[
\nabla V(x)=-\frac{ma}{|x|^{m+1}}\frac{x}{|x|}+O\Big(\frac{1}{|x|^{m+1+\sigma}}\Big),\ \text{as}\ |x|\rightarrow+\infty,
\]
by a similar argument, we get
\[
I_{11}=a_0m|Q_k|^{-m-1}\frac{Q_k}{|Q_k|}++O(R^{-m-1-\sigma})+O(R^{-m-3}).
\]

For the term $I_{12}$, we claim that there exists a constant $C$ (independent of $K$) such that
\begin{align*}
|I_{12}|\leq CR^{-m}e^{-d}d^{-\frac{N-3}{2}}
\leq CR^{-2m}.
\end{align*}
Indeed, by Lemma~\ref{lem-d-R}, Lemma~\ref{lem-3.5} and Lemma~\ref{lem-act}, we have
\begin{align*}
&\Big|\sum_{j\neq k}\int_{\mathbb{R}^N}(V(x)-1)\,w_{Q_j}\nabla w_{Q_k}\,dx\Big|\\
&\leq\Big|\sum_{j\neq k}\int_{|x|<\frac{R}{3}}(V(x)-1)\,w_{Q_j}\nabla w_{Q_k}\,dx\Big|
+\Big|\sum_{j\neq k}\int_{|x|\geq\frac{R}{3}}(V(x)-1)\,w_{Q_j}\nabla w_{Q_k}\,dx\Big|\\
&\leq CKe^{-R/2}+C\sum_{j\neq k}R^{-m}e^{-|Q_j-Q_k|}|Q_j-Q_k|^{-\frac{N-3}{2}}\\
&\leq CR^{-m}e^{-d}d^{-\frac{N-3}{2}}
\leq CR^{-2m}.
\end{align*}

Combining the above estimates, we complete the proof of Claim 1.

\noindent{{\bf Proof of Claim 2:}}
We first divide $I_2$ into two parts:
\begin{align*}
I_2&=\underbrace{\int_{\mathbb{R}^N}pw_{Q_k}^{p-1}\big(\sum_{j\neq k}w_{Q_j}\big)\nabla w_{Q_k}\,dx}_{I_{21}}\\
&\quad+\underbrace{\int_{\mathbb{R}^N}\Big\{\big(\sum_{j=1}^Kw_{Q_j}\big)^p-w_{Q_k}^p-pw_{Q_k}^{p-1}\sum_{j\neq k}w_{Q_j}-\sum_{j\neq k}w_{Q_j}^p\Big\}\nabla w_{Q_k}\,dx}_{I_{22}}.
\end{align*}

By the definition~\eqref{def-Psi} of the interaction function $\Psi$,
\begin{align*}
I_{21}=-\sum_{j\neq k}\Psi(|Q_j-Q_k|)\frac{(Q_j-Q_k)}{|Q_j-Q_k|}.
\end{align*}

For the term $I_{22}$, we divide the domain of integration into $K+1$ parts: $\Omega_1^\ell,\dots,\Omega_{K+1}^\ell$. On $\Omega_{K+1}^\ell$, by taking $\ell$ large but independent of $K$, we have
\begin{align*}
\Big|\big(\sum_{j=1}^Kw_{Q_j}\big)^p-w_{Q_k}^p-pw_{Q_k}^{p-1}\sum_{j\neq k}w_{Q_j}-\sum_{j\neq k}w_{Q_j}^p\Big|\leq CK^{p}e^{-p\ell\rho/2}\leq CK^{-m-3}.
\end{align*}
On $\Omega_l^\ell$, $l=1,\dots,K$, by the definition, we have
\begin{align}\label{eq-5.5}
|x-Q_l|\leq|x-Q_j|,\ \forall\,x\in\Omega_l^\ell,\ \text{and}\ \forall\,1\leq j\leq K.
\end{align}
Since $\ell$ is independent of $K$, for $x\in\Omega_k^\ell$, by mean value theorem and \eqref{ineq-q},
\begin{align*}
\Big|\Big\{\big(\sum_{j=1}^Kw_{Q_j}\big)^p-w_{Q_k}^p-pw_{Q_k}^{p-1}\sum_{j\neq k}w_{Q_j}-\sum_{j\neq k}w_{Q_j}^p\Big\}\nabla w_{Q_k}\Big|
\leq C\Big\{w_{Q_k}^{p-1}\big(\sum_{j\neq k}w_{Q_j}\big)^2+w_{Q_k}\sum_{j\neq k}w_{Q_j}^p\Big\}.
\end{align*}
Note that for $x\in\Omega_k^\ell$, $|x-Q_j|\geq \rho/2$ for all $j\neq k$. Hence by Lemma~\ref{lem-w}, \eqref{eq-5.5} and Lemma~\ref{lem-3.5}, and by choosing $0<\eta<1$ sufficiently small, we have
\begin{align*}
&w_{Q_k}^{p-1}\big(\sum_{j\neq k}w_{Q_j}\big)^2+w_{Q_k}\sum_{j\neq k}w_{Q_j}^p\\
&\leq C\rho^{-\min\{2,p\}\frac{N-1}{2}}\Big\{e^{-(p-1-\eta)|x-Q_k|}\big(\sum_{j\neq k}e^{-|x-Q_j|}\big)^2+e^{-(1-\eta)|x-Q_k|}\sum_{j\neq k}e^{-p|x-Q_j|}\Big\}w_{Q_k}^{\eta}\\
&\leq Ce^{-\min\{2,\frac{p+1-\eta}{2}\}d}d^{-\min\{2,p\}\frac{N-1}{2}}w_{Q_k}^{\eta}.
\end{align*}
Similarly, for all $x\in\Omega_l^\ell$ ($l\neq k$), by mean value theorem, \eqref{ineq-q}, Lemma~\ref{lem-w}, \eqref{eq-5.5} and Lemma~\ref{lem-3.5}, we have
\begin{align*}
\Big|\Big\{\big(\sum_{j=1}^Kw_{Q_j}\big)^p-w_{Q_k}^p-pw_{Q_k}^{p-1}\sum_{j\neq k}w_{Q_j}-\sum_{j\neq k}w_{Q_j}^p\Big\}\nabla w_{Q_k}\Big|
\leq C\Big\{w_{Q_k}w_{Q_l}^{p-1}\sum_{j\neq l}w_{Q_j}+w_{Q_k}^pw_{Q_l}\Big\},
\end{align*}
and
\begin{align*}
&w_{Q_k}w_{Q_l}^{p-1}\sum_{j\neq l}w_{Q_j}+w_{Q_k}^pw_{Q_l}\\
&\leq C\rho^{-\min\{2,p\}\frac{N-1}{2}}\Big\{e^{-(2-\eta)|x-Q_k|}e^{-(p-1)|x-Q_l|}+e^{-(p-\eta)|x-Q_k|}e^{-|x-Q_l|}\\
&\quad\qquad\qquad\qquad\qquad+e^{-(1-\eta)|x-Q_k|}e^{-(p-1)|x-Q_l|}\sum_{j\neq k,l}e^{-|x-Q_j|}\Big\}e^{-\eta|x-Q_k|}\\
&\leq Cd^{-\min\{2,p\}\frac{N-1}{2}}e^{-\min\{2-\eta,\frac{p+1-\eta}{2},p-\eta\}d}e^{-\eta|x-Q_k|}.
\end{align*}
Since $p>1$, we can choose $\eta>0$ such that  $p-\eta>\frac{p+1-\eta}{2}$. Hence, there is a constant (independent of $K$) such that for all $x\in\mathbb{R}^N$,
\begin{align*}
&\Big|\Big\{\big(\sum_{j=1}^Kw_{Q_j}\big)^p-w_{Q_k}^p-pw_{Q_k}^{p-1}\sum_{j\neq k}w_{Q_j}-\sum_{j\neq k}w_{Q_j}^p\Big\}\nabla w_{Q_k}\Big|\\
&\leq Cd^{-\min\{2,p\}\frac{N-1}{2}}e^{-\min\{2-\eta,\frac{p+1-\eta}{2}\}d}e^{-\eta|x-Q_k|}.\end{align*}
Therefore,
\begin{align*}
|I_{22}|\leq Cd^{-\min\{2,p\}\frac{N-1}{2}}e^{-\min\{2-\eta,\frac{p+1-\eta}{2}\}d}
\leq CR^{-\min\{2-\eta,\frac{p+1-\eta}{2}\}m}.
\end{align*}
Combining the above estimates, we complete the proof of Claim 2.
\end{proof}

Now we can analyze $\int_{\mathbb{R}^N}E\,\frac{\partial U}{\partial\textbf{q}}\,dx$. Before we do this, we define
\begin{align*}
\widehat{d}=-\frac{\Psi'(d)}{\Psi(d)}d
=d+O(1),
\end{align*}
and for $j=1,\dots,K$, we denote
\begin{align*}
\bar{f}_j=\dot{f}_j+\dot{f}_{j-1}=(f_{j+1}-f_{j-1})\frac{K}{2\pi},\
\bar{g}_j=\dot{g}_j+\dot{g}_{j-1}=(g_{j+1}-g_{j-1})\frac{K}{2\pi}.
\end{align*}

\begin{lemma}\label{lem-4-2}
Under the assumption of Proposition \ref{prop-3-1}, for sufficiently large $K$, the following expansion holds:
\begin{align*}
-\int_{\mathbb{R}^N}E\,\frac{\partial U}{\partial\textbf{q}}\,dx
&=a_0R^{-m-2}T\textbf{q}
+R^{-m-\sigma}\varPi_{1}(\alpha,\textbf{q})+R^{-m-3}\varPi_{2}(\alpha,\textbf{q})+R^{-2m}\varPi_{3}(\alpha,\textbf{q})\\
&\quad+R^{-\min\{2-\eta,\frac{p+1-\eta}{2}\}m}\varPi_{4}(\alpha,\textbf{q})+R^{-m-3}(\ln K)^{2}\varPi_{5}(\alpha,\textbf{q},\dot{\textbf{q}},\ddot{\textbf{q}}),
\end{align*}
where $\varPi_{1}(\alpha,\textbf{q}),\dots,\varPi_{4}(\alpha,\textbf{q}),\varPi_{5}(\alpha,\textbf{q},\dot{\textbf{q}},\ddot{\textbf{q}})$ are uniformly bounded smooth vector valued functions with $\varPi_{5}(\alpha,0,0,0)=0$, and $T$ is an $2K\times2K$ matrix defined by
\begin{equation}\label{eq-T}
T=\begin{pmatrix}
c_1\,A_1+c_4\,I & c_2\,A_2\\\\
-c_2\,A_2 & c_3\,A_1
\end{pmatrix},
\end{equation}
Here $I$ is the $K\times K$ identity matrix, both $A_1$ and $A_2$ are $K\times K$ circulant matrices given by
\[
 A_1=
\begin{pmatrix}
-2 & 1 & 0 & \cdots & 0 & 1 \\
1 & -2 & 1 & 0 &\cdots & 0 \\
0 & 1 & -2 & 1 & 0 & \cdots\\
\vdots  & \ddots  & \ddots & \ddots  & \ddots & \vdots\\
0 & \cdots & 0 & 1 & -2 & 1\\
1 & 0 & \cdots & 0 & 1 & -2
\end{pmatrix},\
A_2=
\begin{pmatrix}
0 & 1 & 0 & \cdots & 0 & -1 \\
-1 & 0 & 1 & 0 &\cdots & 0 \\
0 & -1 & 0 & 1 & 0 & \cdots\\
\vdots  & \ddots  & \ddots & \ddots  & \ddots & \vdots\\
0 & \cdots & 0 & -1 & 0 & 1\\
1 & 0 & \cdots & 0 & -1 & 0
\end{pmatrix},
\]
and $c_1,c_2,c_3,c_4$ are constants given by
\begin{align}\label{eq-c1-c4}
c_1=\frac{K^2}{4\pi^2},\
c_2=(\widehat{d}-1)\frac{K}{4\pi},\
c_3=-\widehat{d}\,\frac{K^2}{4\pi^2},\
c_4=\widehat{d}-m-1.
\end{align}
\end{lemma}

\begin{proof}
First a simple computation shows that\begin{align*}
|Q_k|^{-m-1}\frac{Q_k}{|Q_k|}&=|Q_k^0+q_k|^{-m-1}\frac{(Q_k^0+q_k)}{|Q_k^0+q_k|}\\
&=R^{-m-1}\vec{n}_k+R^{-m-2}\Big\{g_k\vec{t}_k-(m+1)f_k\vec{n}_k\Big\}+O(R^{-m-3}).
\end{align*}

To estimate $I_2$, by direct computation, we have
\begin{align*}
&\frac{Q_{j+1}-Q_j}{|Q_{j+1}-Q_j|}\\
&=\Big\{-\sin\frac{\pi}{K}+(\dot{f}_j-g_j)R^{-1}-\dot{g}_jR^{-1}\frac{\pi}{K}-(\dot{f}_j-g_j)(f_j+\dot{g}_j)R^{-2}\Big\}\vec{n}_j\\
&\quad+\Big\{\cos\frac{\pi}{K}+(\dot{f}_j-g_j)R^{-1}\frac{\pi}{K}-\frac{1}{2}(\dot{f}_j-g_j)^2R^{-2}\Big\}\vec{t}_j+O(K^{-3}(\ln K)^{-1}),
\end{align*}
\begin{align*}
&\frac{Q_{j-1}-Q_j}{|Q_{j-1}-Q_j|}\\
&=\Big\{-\sin\frac{\pi}{K}+(-\dot{f}_{j-1}+g_j)R^{-1}-\dot{g}_{j-1}R^{-1}\frac{\pi}{K}-(-\dot{f}_{j-1}+g_j)(f_j+\dot{g}_{j-1})R^{-2}\Big\}\vec{n}_j\\
&\quad-\Big\{\cos\frac{\pi}{K}+(-\dot{f}_{j-1}+g_j)R^{-1}\frac{\pi}{K}-\frac{1}{2}(-\dot{f}_{j-1}+g_j)^2R^{-2}\Big\}\vec{t}_j+O(K^{-3}(\ln K)^{-1}),
\end{align*}
and
\begin{align*}
&\Psi(|Q_{j+1}-Q_j|)\\
&=\Psi(d+(2f_j+\bar{g}_j)\frac{\pi}{K})
+\Psi'(d)\Big\{2\ddot{g}_j\frac{\pi^2}{K^2}+2\dot{f}_j\frac{\pi^2}{K^2}+(\dot{f}_j-g_j)^2R^{-1}\frac{\pi}{K}\Big\}\\
&\quad+\Psi(d)O(K^{-3}),
\end{align*}
\begin{align*}
&\Psi(|Q_{j-1}-Q_j|)\\
&=\Psi(d+(2f_j+\bar{g}_j)\frac{\pi}{K})
+\Psi'(d)\Big\{-2\ddot{g}_j\frac{\pi^2}{K^2}-2\dot{f}_{j-1}\frac{\pi^2}{K^2}+(-\dot{f}_{j-1}+g_j)^2R^{-1}\frac{\pi}{K}\Big\}\\
&\quad+\Psi(d)O(K^{-3}).
\end{align*}
Therefore,
\begin{align*}
&\sum_{j\in\{k-1,\,k+1\}}\Psi(|Q_j-Q_k|)\frac{(Q_j-Q_k)}{|Q_j-Q_k|}\\
&=-2\sin\frac{\pi}{K}\Psi(d)\vec{n}_j
+2\Psi'(d)\frac{\pi^2}{K^2}\Big\{-(2f_j+\bar{g}_j)\vec{n}_j+(\bar{f}_j+2\ddot{g}_j)\vec{t}_j\Big\}\\
&\quad+\Psi(d)R^{-1}\frac{\pi}{K}\Big\{(2\ddot{f}_j-\bar{g}_j)\vec{n}_j+(\bar{f}_j-2g_j)\vec{t}_j\Big\}+\Psi(d)O(K^{-3}).
\end{align*}

Combining the above estimates and Lemma \ref{lem-4-1}, we get
\begin{align*}
\int_{\mathbb{R}^N}E\mathcal{Z}_{Q_k}\,dx&=a_0R^{-m-2}\Big\{-(m+1)f_k+(\ddot{f}_k-\frac{1}{2}\bar{g}_k)+\widehat{d}(f_k+\frac{1}{2}\bar{g}_k)\Big\}\vec{n}_k\\
&\quad+a_0R^{-m-2}\Big\{g_k+(\frac{1}{2}\bar{f}_k-g_k)-\widehat{d}(\frac{1}{2}\bar{f}_k+\ddot{g}_k)\Big\}\vec{t}_k\\
&\quad+R^{-m-\sigma}\varPi_{k,1}(\alpha,\textbf{q})+R^{-m-3}\varPi_{k,2}(\alpha,\textbf{q})+R^{-2m}\varPi_{k,3}(\alpha,\textbf{q})\\
&\quad+R^{-\min\{2-\eta,\frac{p+1-\eta}{2}\}m}\varPi_{k,4}(\alpha,\textbf{q})+R^{-m-3}(\ln K)^{2}\varPi_{k,5}(\alpha,\textbf{q},\dot{\textbf{q}},\ddot{\textbf{q}})
\end{align*}
where $\varPi_{k,l}(\alpha,\textbf{q})$'s and $\varPi_{k,5}(\alpha,\textbf{q},\dot{\textbf{q}},\ddot{\textbf{q}})$ are smooth vector valued functions, which are uniformly bounded as $K\rightarrow\infty$. Moreover, $\varPi_{k,5}(\alpha,0,0,0)=0$. The desired result follows.
\end{proof}

Next we compute $\int_{\mathbb{R}^N}(L[\phi]+N(\phi))\frac{\partial U}{\partial\textbf{q}}\,dx$.

\begin{lemma}\label{lem-4-3}
Under the assumption of Proposition \ref{prop-3-1}, for sufficiently large $K$, the following expansions hold true:
\begin{equation*}
\int_{\mathbb{R}^N}L[\phi]\frac{\partial U}{\partial\textbf{q}}\,dx
=K^{-\min\{2,\frac{p}{2}+1-\frac{\eta}{2},p-\frac{\eta}{2}\}m}(\ln K)^{2}\varPi_{6}(\alpha,\textbf{q}),
\end{equation*}
and
\begin{equation*}
\int_{\mathbb{R}^N}N(\phi)\frac{\partial U}{\partial\textbf{q}}\,dx
=K^{-\min\{2,\,p-\eta\}m}(\ln K)^{-1}\varPi_{7}(\alpha,\textbf{q}),
\end{equation*}
where $\varPi_{6}(\alpha,\textbf{q}),\varPi_{7}(\alpha,\textbf{q})$ are uniformly bounded smooth vector valued functions.
\end{lemma}

\begin{proof}
By integration by parts, \eqref{LZ} and Proposition~\ref{prop-3-1}, we have
\begin{align*}
\Big|\int_{\mathbb{R}^N}L[\phi]\,\mathcal{Z}_{Q_k}\,dx\Big|
=\Big|\int_{\mathbb{R}^N}\phi\,L[\mathcal{Z}_{Q_k}]\,dx\Big|
&\leq Cde^{-\min\{1,\frac{p}{2}\}d}\|\phi\|_{L^\infty(\mathbb{R}^N)}\\
&\leq CK^{-\min\{2,\frac{p}{2}+1-\frac{\eta}{2},p-\frac{\eta}{2}\}m}(\ln K)^{2}.
\end{align*}

For the second estimate, since $\|\phi\|_{**}\leq CK^{-\min\{1,\frac{p-\eta}{2}\}m}(\ln K)^{-\frac{1}{2}}$, we have
\begin{align*}
|N(\phi)|\leq
\begin{cases}
CU^{p-2}|\phi|^2,&\text{for}\ |\phi|\leq U/2,\\
C|\phi|^{p},&\text{for}\ |\phi|\geq U/2.
\end{cases}
\end{align*}
We claim that
\begin{align*}
\Big|\int_{\mathbb{R}^N}N(\phi)\mathcal{Z}_{Q_k}\,dx\Big|
\leq C\|\phi\|_{**}^2
\leq CK^{-\min\{2,\,p-\eta\}m}(\ln K)^{-1}.
\end{align*}
Indeed, when $p\geq2$, this follows from $|N(\phi)|\leq C|\phi|^2$. Now we consider the case  $p<2$. In this case, it is not hard to get $|N(\phi)|\leq CU^{p-2}|\phi|^2$. Since $U\geq w_{Q_k}$, we have
\begin{align*}
|N(\phi)\mathcal{Z}_{Q_k}|\leq CU^{p-2}w_{Q_k}|\phi|^2
\leq Cw_{Q_k}^{p-1}|\phi|^2,\ \text{if}\ p<2,
\end{align*}
from which we get the desired result.
\end{proof}

\subsection{The invertibility of $T$}

In this subsection, we study the linear problem $T\textbf{q}=\textbf{b}$ and get the following result, whose proof is delayed to Appendix A.

\begin{lemma}\label{lem-T}
There is an $K_0\in\mathbb{N}$ such that for all $K\geq K_0$ and every $\textbf{b}\in\mathbb{R}^{2K}$, there exists a unique vector $\vec{\textbf{q}}\in\mathbb{R}^{2K}$ and a unique constant $\gamma\in\mathbb{R}$ such that
\begin{equation}
T\vec{\textbf{q}}=\textbf{b}+\gamma\,\textbf{q}_1,\ \vec{\textbf{q}}\perp\,\textbf{q}_0.
\end{equation}
Moreover, there is a positive constant $C$ which is independent of $K$ such that
\begin{equation}\label{q-b-0}
\|\vec{\textbf{q}}\|_*\leq C(\ln K)^2\|\textbf{b}\|_\infty.
\end{equation}
\end{lemma}

Denote the inverse of $T$ in the sense of Lemma~\ref{lem-T} by $T^{-1}$. Since $\textbf{q}_1$ depends on the parameter $\textbf{q}$, the matrix $T^{-1}$ depends on $\textbf{q}$ and thus we write $T^{-1}=T^{-1}_{\textbf{q}}$.

\subsection{Reduction to one dimension}

Now we can state the main result in this section.

\begin{proposition}\label{prop-4-1}
Under the assumption of Theorem \ref{thm-0}, there is an integer $K_0>0$ such that: for all integer $K\geq K_0$ and for each $\alpha\in\mathbb{R}$, there exists a unique $(\textbf{q},\gamma)=(\textbf{q}(\alpha),\gamma(\alpha))$ such that $\vec{\beta}(\alpha,\textbf{q},\gamma)=0$. As a result, $\phi(x;\alpha,\textbf{q}(\alpha))$ and $\gamma(\alpha)$ satisfy the equation:
\begin{equation}\label{eq-u-2-a}
\begin{cases}
L[\phi]+E+N(\phi)=\gamma\frac{\partial U}{\partial\alpha},
\vspace{1mm}\\
\int_{\mathbb{R}^N}\phi\,\mathcal{Z}_{Q_j}\,dx=0,\ \forall\,j=1,\dots,K.
\end{cases}
\end{equation}
Moreover, the function $\phi(x;\alpha,\textbf{q}(\alpha))$ is of class $C^1$ in $\alpha$, and we have
\begin{equation}
\|\phi\|_{**}\leq C_0K^{-\min\{1,\frac{p-\eta}{2}\}m}(\ln K)^{-\frac{1}{2}},\
\|\textbf{q}\|_*+R^{-1}\|\partial_{\alpha}\textbf{q}\|_*\leq CK^{-\mu}(\ln K)^2,
\end{equation}
where $C$ is a positive constant (independent of $K$), and
\begin{align}
0<\mu<\min\left\{\sigma-2,\,\min\big\{1-\eta,\frac{p-1-\eta}{2}\big\}m-2\right\}.
\end{align}
\end{proposition}

To prove Proposition \ref{prop-4-1}, it suffices to solve $\vec{\beta}(\alpha,\textbf{q},\gamma)=0$ for each $\alpha$. By the results in the preceding subsections, we can rewrite this equation in a more explicit form.
\begin{lemma}\label{lem-4-3}
For every $\alpha\in\mathbb{R}$, the equation $\vec{\beta}(\alpha,\textbf{q},\gamma)=0$ is equivalent to
\begin{align}\label{q-4-1}
-a_0mR^{-m-2}T\textbf{q}+\Phi(\alpha,\textbf{q})=\gamma\,\textbf{q}_1,
\end{align}
where $T$ is the $2K\times2K$ matrix defined in \eqref{eq-T}, $\Phi$ denotes the remaining term, and
\begin{align*}
\textbf{q}_1=\int_{\mathbb{R}^N}\frac{\partial U}{\partial \alpha}\frac{\partial U}{\partial\textbf{q}}\,dx=M(R\textbf{q}_0+\textbf{q}^\perp).
\end{align*}
\end{lemma}

By Lemma~\ref{lem-4-2}, Lemma~\ref{lem-4-3} and Lemma~\ref{lem-4-3}, we have the following estimate of $\Phi(\alpha,\textbf{q})$.
\begin{lemma}\label{lem-Phi}
Under the assumption of Proposition \ref{prop-3-1}, for $K$ sufficiently large, the following expansion holds:
\begin{align*}
\Phi(\textbf{q})&=R^{-m-\sigma}\varPi_{1}(\alpha,\textbf{q})+R^{-m-3}\varPi_{2}(\alpha,\textbf{q})+R^{-2m}\varPi_{3}(\alpha,\textbf{q})\\
&\quad+R^{-\min\{2-\eta,\frac{p+1-\eta}{2}\}m}\varPi_{4}(\alpha,\textbf{q})+R^{-m-3}(\ln K)^{2}\varPi_{5}(\alpha,\textbf{q},\dot{\textbf{q}},\ddot{\textbf{q}})\\
&\quad+K^{-\min\{2,\frac{p}{2}+1-\frac{\eta}{2},p-\frac{\eta}{2}\}m}(\ln K)^{2}\varPi_{6}(\alpha,\textbf{q})
+K^{-\min\{2,\,p-\eta\}m}(\ln K)^{-1}\varPi_{7}(\alpha,\textbf{q}),
\end{align*}
where $\varPi_j(\alpha,\textbf{q})$'s and $\varPi_{5}(\alpha,\textbf{q},\dot{\textbf{q}},\ddot{\textbf{q}})$ are smooth vector valued functions, which are uniformly bounded as $K\rightarrow\infty$. Moreover, $\varPi_{5}(\alpha,0,0,0)=0$.
\end{lemma}

Now we are going to solve \eqref{q-4-1} and then complete the proof of Proposition \ref{prop-4-1}.

\begin{proof}[Proof of Proposition~\ref{prop-4-1}]
By Lemma \ref{lem-T}, equation \eqref{q-4-1} is equivalent to
\begin{align*}
\textbf{q}=(a_0m)^{-1}T^{-1}_{\textbf{q}}\left[R^{m+2}\Phi(\alpha,\textbf{q})\right]=\mathcal{F}(\textbf{q}).
\end{align*}
By Lemma \ref{lem-Phi}, since $\min\{2-\eta,\frac{p+1-\eta}{2}\}\leq\min\{2,\,p-\eta\}\leq\min\{2,\frac{p}{2}+1-\frac{\eta}{2},p-\frac{\eta}{2}\}$ for $0<\eta<p-1$, we get
\begin{align*}
(a_0m)^{-1}R^{m+2}\Phi(\alpha,\textbf{q})=K^{-\mu}\widetilde{\varPi}(\alpha,\textbf{q})+(K^{-1}\ln K)\widetilde{\varXi}(\alpha,\textbf{q},\dot{\textbf{q}},\ddot{\textbf{q}}),
\end{align*}
where both $\widetilde{\varPi}$ and $\widetilde{\varXi}$ are smooth vector valued functions, which are uniformly bounded as $K$ tends to infinity. Moreover, $\widetilde{\varXi}(\alpha,0,0,0)=0$.

Hence by Lemma \ref{lem-T}, for $\|\textbf{q}\|_*\leq1/2$, we have
\begin{align*}
\|\mathcal{F}(\textbf{q})\|_*
\leq C\left(K^{-\mu}(\ln K)^2+K^{-1}(\ln K)^3\right)
\leq CK^{-\mu}(\ln K)^2,
\end{align*}
and
\begin{align*}
\|\mathcal{F}(\textbf{q})-\mathcal{F}(\mathring{\textbf{q}})\|_*
&\leq C\left(K^{-\mu}(\ln K)^2+K^{-1}(\ln K)^3\right)\|\textbf{q}-\mathring{\textbf{q}}\|_*
\leq\frac{1}{2}\|\textbf{q}-\mathring{\textbf{q}}\|_*,
\end{align*}
since $\|(T^{-1}_{\textbf{q}}-T^{-1}_{\mathring{\textbf{q}}})\textbf{b}\|_*\leq CKR^{-m}(\ln K)^2\|\textbf{b}\|_\infty\|\textbf{q}-\mathring{\textbf{q}}\|_*$. Therefore, $\mathcal{F}$ is a contraction mapping. By the Banach fixed point theorem, the result follows.

To show the differentiability of $\textbf{q}(\alpha)$. Consider the map $\mathcal{T}(\alpha,\textbf{q})=\textbf{q}-\mathcal{F}(\alpha,\textbf{q}):\mathbb{R}\times\mathbb{R}^{2K}\mapsto\mathbb{R}^{2K}$ of class $C^1$. Since $\frac{\partial\mathcal{F}}{\partial\textbf{q}}=O(K^{-\mu-1}(\ln K)^2)$, $\frac{\partial\mathcal{T}}{\partial\textbf{q}}\big|_{\alpha,\textbf{q}(\alpha)}=I-\frac{\partial\mathcal{F}}{\partial\textbf{q}}(\alpha,\textbf{q}(\alpha))$ is invertible, we get the differentiability of $\textbf{q}(\alpha)$.

Next we study the dependence of $\textbf{q}$ on $\alpha$. Assume that we have two solutions corresponding to two sets of parameters. One of them denoted by
\begin{align*}
\textbf{q}=(a_0m)^{-1}T^{-1}_{\textbf{q}}\left[R^{m+2}\Phi(\alpha,\textbf{q})\right],
\end{align*}
corresponds to $\alpha$; the other denoted by
\begin{align*}
\mathring{\textbf{q}}=(a_0m)^{-1}T^{-1}_{\mathring{\textbf{q}}}\left[R^{m+2}\Phi(\mathring{\alpha},\mathring{\textbf{q}})\right],
\end{align*}
corresponds to $\mathring{\alpha}$.
Assume that $R|\mathring{\alpha}-\alpha|\leq1/2$, by a direct computation and Lemma \ref{lem-T}, we have
\begin{equation*}
\|\textbf{q}-\mathring{\textbf{q}}\|_*\leq CK^{-\mu}(\ln K)^2(R|\mathring{\alpha}-\alpha|),
\end{equation*}
from which we get the desired result.
\end{proof}


\section{Proof of Theorem \ref{thm-0}: variational reduction}

In this section, our purpose is to achieve Step 2.B in the setting up of the problem and then complete the proof of Theorem \ref{thm-2}.

To solve $\gamma(\alpha)=0$ in Step 2.B, we will apply the variational reduction. To do this, we first introduce some notation. Let $\alpha\in\mathbb{R}$ and $\phi=\phi(x;\alpha,\textbf{q}(\alpha))$ be the function given in Proposition \ref{prop-4-1}, we define the reduced energy function by
\begin{equation}\label{def-F}
F(\alpha)=\mathcal{E}(U+\phi):\mathbb{R}\rightarrow\mathbb{R},
\end{equation}
where we write $U=U(x;\alpha,\textbf{q}(\alpha))$.

By \eqref{Q-1}, both $U$ and $\phi$ are $2\pi$ periodic in $\alpha$. Hence by Proposition \ref{prop-4-1}, the reduced energy function $F(\alpha)$ has the following property.
\begin{lemma}\label{lem-5-1}
The function $F(\alpha)$ is of class $C^1$ and satisfies $F(\alpha+2\pi)=F(\alpha)$ for every $\alpha\in\mathbb{R}$.
\end{lemma}

With this notation, the next lemma shows that if $F(\alpha)$ has a critical point then $\gamma(\alpha)=0$ has a solution. In other words, after the Lyapunov-Schmidt reduction, the following lemma concerns the relation between the critical points of $F(\alpha)$ and those of the energy functional $\mathcal{E}(u)$.

\begin{lemma}\label{lem-5-2}
Under the assumption of Proposition \ref{prop-4-1}, there exists $K_0\in\mathbb{N}_+$ such that: for all integer $K\geq K_0$, if $\alpha_0$ be a critical point of $F(\alpha)$, then $\gamma(\alpha_0)=0$ and the corresponding function
\begin{align*}
u(x)=U(x;\alpha_0,\textbf{q}(\alpha_0))+\phi(x;\alpha_0,\textbf{q}(\alpha_0))
\end{align*}
is a solution of \eqref{eq-u}.
\end{lemma}

\begin{proof}
By Proposition \ref{prop-4-1}, for $K$ sufficiently large and for every $\alpha\in\mathbb{R}$, $\phi=\phi(x;\alpha,\textbf{q}(\alpha))$ satisfies the equation
\begin{equation}\label{eq-5-1}
S(U+\phi)=\gamma(\alpha)\frac{\partial U}{\partial\alpha}.
\end{equation}

By the definition \eqref{def-F}, we obtain
\begin{align*}
F'(\alpha)=\int_{\mathbb{R}^N}S(U+\phi)(\partial_{\alpha}U+\partial_{\alpha}\phi)\,dx,
\end{align*}
where $\partial_{\alpha}U=\frac{\partial U}{\partial\alpha}+\frac{\partial U}{\partial\textbf{q}}\cdot\partial_\alpha\textbf{q}$ and $\partial_{\alpha}\phi=\frac{\partial \phi}{\partial\alpha}+\frac{\partial \phi}{\partial\textbf{q}}\cdot\partial_\alpha\textbf{q}$. Thus by using \eqref{eq-5-1},
\begin{align*}
F'(\alpha)=\gamma(\alpha)\int_{\mathbb{R}^N}\frac{\partial U}{\partial\alpha}(\partial_{\alpha}U+\partial_{\alpha}\phi)\,dx.\end{align*}
If $\alpha_0$ be a critical point of $F(\alpha)$, then $F'(\alpha_0)=0$. Hence to prove $\gamma(\alpha_0)=0$, it is sufficient to show that
\begin{align}\label{eq-5-4}
\int_{\mathbb{R}^N}\frac{\partial U}{\partial\alpha}(\partial_{\alpha}U+\partial_{\alpha}\phi)\,dx\neq 0.
\end{align}
In fact, by Proposition \ref{prop-4-1} and Proposition~\ref{prop-4-1}, we have
\begin{equation}
\partial_{\alpha}U+\partial_{\alpha}\phi=\frac{\partial U}{\partial\alpha}+\frac{\partial U}{\partial\textbf{q}}\cdot\partial_\alpha\textbf{q}+\frac{\partial \phi}{\partial\alpha}+\frac{\partial \phi}{\partial\textbf{q}}\cdot\partial_\alpha\textbf{q}.
\end{equation}
Recall that
\begin{align*}
\frac{\partial U}{\partial \alpha}&
=(R\textbf{q}_0+\textbf{q}^\perp)\cdot\frac{\partial U}{\partial\textbf{q}},
\end{align*}
hence by Proposition \ref{prop-3-1} and Proposition~\ref{prop-4-1}, we have
\begin{align*}
K^{-1}R^{-2}\int_{\mathbb{R}^N}\frac{\partial U}{\partial\alpha}(\partial_{\alpha}U+\partial_{\alpha}\phi)\,dx=(1+o(1))\int_{\mathbb{R}^N}(\partial_{x_1}w)^2\,dx,
\end{align*}
which implies \eqref{eq-5-4} and completes the proof.

\end{proof}

\begin{proof}[Proof of Theorems \ref{thm-2}]
By Lemma \ref{lem-5-1}, $F(\alpha)$ is $2\pi$ periodic and of class $C^1$. Hence it has at least two critical points (maximum and minimum points) in $[0,2\pi)$. Therefore, Theorem \ref{thm-2} follows from Lemma \ref{lem-5-2}.
\end{proof}


\section{Generalizations and discussion}

In this section we first give some slight extensions of the results proved in the previous sections. Finally we would like to discuss some related questions we do not answer in this paper.

\subsection{More general nonlinearities}

Unlike the minimization method, we do not use the homogeneous property of the nonlinearity of equation \eqref{eq-u}. Therefore, our argument can be applied to construct infinitely many positive solutions for a more general problem:
\begin{align*}
\left\{
\begin{array}{ll}
-\Delta u+V(x) u-f(u)=0\ \text{in}\ \mathbb{R}^N,
\vspace{1mm}\\
u>0\ \text{in}\ \mathbb{R}^N,\ u\in H^1(\mathbb{R}^N),
\end{array}\right.
\end{align*}
where $f:\mathbb{R}\rightarrow\mathbb{R}$ is at least $C^{1,\nu}(\mathbb{R})$ for some $\nu\in(0,1)$, and satisfies the following conditions:
\begin{enumerate}[($f_1$)]
\item $f(u)=0$ for $u\leq0$, $f(0)=f'(0)=0$;
\item The equation
\begin{align*}
\left\{
\begin{array}{ll}
-\Delta u+V_\infty u-f(u)=0\ \text{in}\ \mathbb{R}^N,
\vspace{1mm}\\
u>0\ \text{in}\ \mathbb{R}^N,\ u\in H^1(\mathbb{R}^N),
\end{array}\right.
\end{align*}
has a nondegenerate solution $w$, in the sense that
\[
\text{Ker}\left(-\Delta+1-f'(w)\right)\cap L^\infty(\mathbb{R}^N)=\text{Span}\left\{\partial_{x_1}w,\cdots,\partial_{x_N}w\right\}.
\]
\end{enumerate}

\begin{remark}
It is not hard to see that our argument can be used to deal with the homogeneous Dirichlet boundary condition problem of \eqref{eq-u} in $\mathbb{R}^N\setminus\overbar{\Omega}$, where $\Omega$ is a bounded domain in $\mathbb{R}^N$.
\end{remark}

\subsection{Sign-changing solutions} Suppose that \eqref{cond-wy} holds for some constants
\begin{align*}
V_\infty>0,\ a<0,\ \text{and}\ \min\big\{1,\frac{p-1}{2}\big\}m>2,\ \sigma>2.
\end{align*}
If $N\geq3$, we further assume \eqref{cond-even-1}, using almost the same argument, our method can be applied to construct infinitely many sign-changing solutions of the problem
\begin{align}
-\Delta u+V(x) u-|u|^{p-1}u=0\ \text{in}\ \mathbb{R}^N,\ u\in H^1(\mathbb{R}^N).
\end{align}
A similar result can be found in \cite{CDS-05} when $V(x)$ tends to $V_\infty$ from below with a suitable rate. We emphasize that our method can be applied to a more general non-even nonlinearity.

\subsection{Remarks on condition~\eqref{cond-m}}
In this subsection we consider the possible ways to improve the condition~\eqref{cond-m}. Recall that the key step in our method is to solve the following equation:
\begin{align*}
-a_0mR^{-m-2}T\textbf{q}+\Phi(\alpha,\textbf{q})=\gamma\,\textbf{q}_1,
\end{align*}
the property of $T$ has been described in Lemma~\ref{lem-T}. We pose the condition~\eqref{cond-m} such that $\Phi(\alpha,\textbf{q})$ is a smaller term comparing to $R^{-m-2}T\textbf{q}$. Hence, to refine the condition~\eqref{cond-m}, the estimate of $\Phi(\alpha,\textbf{q})$ is the key point. A better estimate on $\Phi(\alpha,\textbf{q})$ gives a weaker condition on $m$ and $\sigma$. For example, for $N=2$, if we assume that the following asymptotic behaviour of $V$ holds in the $C^1$ sense:
\begin{equation*}
V(x)=V_\infty+\frac{a}{|x|^m}+\frac{a_1(\theta)}{|x|^{m+1}}+O\Big(\frac{1}{|x|^{m+1+\sigma_1}}\Big),\ \text{as}\ |x|\rightarrow\infty,
\end{equation*}
where $a_1(\theta)$ is an $2\pi$ periodic smooth function. Here $(r,\theta)$ is the polar coordinate.  Then ``$\sigma>2$" in the condition~\eqref{cond-m} can be improved to be ``$\sigma_1>0$".

To get a better estimate of $\Phi(\alpha,\textbf{q})$, an improvement of approximation is needed. Recall that
\begin{align*}
E
&=\underbrace{\sum_{j=1}^K\left(V(x)-1\right)w_{Q_j}}_{E_1}-\underbrace{\Big\{\Big(\sum_{j=1}^Kw_{Q_j}\Big)^p-\sum_{j=1}^Kw_{Q_j}^p\Big\}}_{E_2}.
\end{align*}
The leading term of $E_1$ is given by $\sum_{j=1}^K\left(V(Q_j)-1\right)w_{Q_j}$, which is $O(R^{-m})$ by \eqref{cond-wy}. Moreover, it is known that $\varphi_0=-\frac{1}{p-1}w-\frac{1}{2}x\cdot\nabla w$ is the explicit solution of
\begin{align*}
L_0[\varphi_0]=-\Delta\varphi_0+\varphi_0-pw^{p-1}\varphi_0=w\ \text{in}\ \mathbb{R}^N,\ \text{and}\ \int_{\mathbb{R}^N}\varphi_0\nabla w\,dx=0.
\end{align*}
Thanks to the polynomial decay of $V$, we can improve the approximation and write $E_1$ into two parts: one is $O(R^{-m})$ having explicit form; and the other one is $O(R^{-m-1})$, which is enough for a better estimate of the part in $\Phi(\alpha,\textbf{q})$ from $E_1$.

However, for the term $E_2$, the situation is more difficult. Recall that in the proof of Claim 2 of Lemma~\ref{lem-error}, we show that $E_2\sim d^{-\frac{N-1}{2}}e^{-\min\{1,\frac{p}{2}\}d}$. Since $w(r)\sim r^{-\frac{N-1}{2}}e^{-r}$, subtracting the terms of order $d^{-\frac{N-1}{2}}e^{-\min\{1,\frac{p}{2}\}d}$, the next term  is $O(d^{-\frac{N-1}{2}-1}e^{-\min\{1,\frac{p}{2}\}d})$. It is not enough for a better estimate of the part in $\Phi(\alpha,\textbf{q})$ from $E_2$.

Except for getting a better estimate of $\Phi(\alpha,\textbf{q})$, another feasible way to improve the condition~\eqref{cond-m} is to apply min-max theorems to study the reduced energy functional $\mathcal{E}(U)$ and its small perturbations. A useful property is that the matrix $T$ has only one zero eigenvalue, $K-1$ negative eigenvalues, and $K$ positive eigenvalues.

\subsection{The anisotropic case}

Observe that the leading term in \eqref{cond-wy} is radial. Thus it is interesting to consider the fully anisotropic case of the problem \eqref{eq-u} for $N=2$.
\begin{openproblem}\label{question-1}
Do there still exist infinitely many positive solution of problem \eqref{eq-u} if there are constants $V_\infty>0$, $m>1$, $\sigma>0$ and a positive $2\pi$ periodic smooth function $a(\theta)$ such that
\begin{equation}\label{cond-dwy}
V(x)=V_\infty+\frac{a(\theta)}{|x|^m}+O\Big(\frac{1}{|x|^{m+\sigma}}\Big),\ \text{as}\ |x|\rightarrow\infty.
\tag{$V3$}
\end{equation}
Here $(r,\theta)$ is the polar coordinate.
\end{openproblem}

Inspired by the results in \cite{CDS-11,AW-12}, the answer is very likely yes. But in this situation, the idea of uniformly distribution of points on curves does not work. Indeed, let $\Gamma=(r(\theta),\theta)$ be a closed curve in the polar coordinate system. Denote its length by $L$ and its natural parameter by $s(\theta)$. If one puts $K$ points on the stretched curve $R\,\Gamma$ for some positive constant $R$. After some computations, we get the balancing condition on $\Gamma$ and $R$:
\begin{align}\label{sys-1}
\left\{\begin{array}{ll}
\big(\frac{1}{2}\int_{\mathbb{R}^N}w^2\,dy\big)R^{-m-1}=C\Psi\big(\frac{RL}{K}\big)\frac{L}{K},
\vspace{1mm}\\
C\Big\{ma(\theta(s))\gamma(s)-a'(\theta(s))\gamma^\perp(s)\Big\}|\gamma(s)|^{-m-2}+\gamma''(s)=0,
\end{array}
\right.
\end{align}
where $C>0$ is a constant, $L$ is the length of $\gamma$ and $\theta(s)$ is the inverse of $s(\theta)$. However, it can be proved that system~\eqref{sys-1} has a solution if and only if $a(\theta)\equiv\text{constant}$. Therefore, it is to be expected that the spikes cannot be uniformly distributed on a closed curve if $a(\theta)$ is not a constant function.

A feasible way to answer Question~\ref{question-1} is to develop a theory like \cite{AMPW-12}. But a more accurate reduction procedure would be required since the mutual angles between the adjacent rays connecting spikes and the origin goes to zero as $K$ tends to infinity.

\subsection{Optimal condition on the decay}

It seems that our argument here can only deal with the case of polynomial decay. Inspired by \cite{CDS-11,AW-12}, it is reasonable to believe that there are infinitely many positive solutions when the potential $V$ satisfies the following decay assumption:
\begin{equation}\label{cond-CDS}
\exists\,\bar{\eta}\in(0,\sqrt{V_\infty}):\ \lim_{|x|\rightarrow\infty}\left(V(x)-V_\infty\right)e^{\bar{\eta}|x|}=+\infty.
\tag{$V4$}
\end{equation}
Hence it is natural to ask the following question:
\begin{openproblem}\label{question-2}
Does $\mathcal{S}_V\neq\emptyset$ for any potential $V$ satisfying \eqref{V-infinity}? Is the condition~\eqref{cond-CDS} sufficient and necessary for $\#(\mathcal{S}_V)=\infty$, i.e., on the existence of infinitely many positive solution of problem \eqref{eq-u}?
\end{openproblem}


\subsection{Higher dimensions}

For the higher dimension case, i.e., $N\geq3$, as we have seen, our arguments still work under the weak symmetry condition~\eqref{cond-even-1}. It is natural to ask the following question:
\begin{openproblem}\label{question-3}
Does the week symmetry condition~\eqref{cond-even-1} can be dropped when $N\geq3$? Do the points can be distributed in a higher dimensional set, such as spheres in $\mathbb{R}^3$?
\end{openproblem}

\subsection{Higher dimensional concentration phenomena}

Next we turn to the higher dimensional concentration phenomena. Inspired by the results in \cite{AMN-03,dPKW-07,MMM-09,WWY-11} and \cite{M-09,BW-10}, it is interesting to ask the following question:

\begin{openproblem}\label{question-4}
Does there exist solution of problem \eqref{eq-u} concentrating on higher dimensional sets, e.g., curves? That is, does the Ambrosetti-Malchiodi-Ni conjecture in \cite{AMN-03} still hold without the small parameter $\varepsilon$, even in the radial symmetry case? If the answer is yes, are the solutions constructed in Theorem~\ref{thm-2} bifurcations sets?
\end{openproblem}

Under the condition~\eqref{V-infinity}, Question~\ref{question-4} is not easy even in the radial symmetry case. For example, assuming that $N=2$ and $V(x)$ is radially symmetric, if one try to construct a positive solution concentrating on a circle with radius $R$, a simple computation gives $V'(R)\sim1/R$, which is incompatible with $\lim_{|x|\rightarrow\infty}V(x)=V_\infty$.


\section{Appendix A: Circulant matrices and proof of Lemma \ref{lem-T}}

In this section we will prove Lemma \ref{lem-T}. To this end, we need some notation. Denote the $K$-dimensional complex vector space and the ring of $K\times K$ complex matrices by $\mathbb{C}^K$ and $\mathbb{M}_K$, respectively. Let $\textbf{b}=(b_1,b_2,\dots,b_K)\in\mathbb{C}^K$, we define a shift operator $\mathcal{S}:\mathbb{C}^K\rightarrow\mathbb{C}^K$ by
\begin{align*}
\mathcal{S}(b_1,b_2,\dots,b_K)=(b_K,b_1,\dots,b_{K-1}).
\end{align*}

\begin{definition} [cf. \cite{KS-12}]
The circulant matrix $B=\textit{circ}\{\mathbf{b}\}$ associated to the vector $\textbf{b}=(b_1,b_2,\dots,b_K)\in\mathbb{C}^K$ is the $K\times K$ matrix whose $n$th row is $\mathcal{S}^{n-1}\textbf{b}$:
\[
B=
\begin{pmatrix}
b_1 & b_2 & \cdots & b_{K-1} & b_K \\
b_K & b_1 &\cdots & b_{K-2} & b_{K-1} \\
\vdots  & \vdots  & \ddots & \vdots & \vdots\\
b_3 & b_4 & \cdots & b_1 & b_2\\
b_2 & b_3 & \cdots & b_K & b_1
\end{pmatrix}.
\]
We denote by $\textit{Circ}(K)\subset\mathbb{M}_K$ the set of all $K\times K$ complex circulant matrices.
\end{definition}

With this notation, both $A_1$ and $A_2$ in \eqref{eq-T} are $K\times K$ circulant matrices. In fact,
\begin{align*}
A_1=\textit{circ}\big\{(-2,1,0,\dots,0,1)\big\}\ \text{and}\
A_2=\textit{circ}\big\{(0,1,0,\dots,0,-1)\big\}.
\end{align*}
Let $\epsilon=e^{i\frac{2\pi}{K}}$ be a primitive $K$-th root of unity, we define
\begin{align*}
X_l=\frac{1}{\sqrt{K}}(1,\epsilon^{l-1},\epsilon^{2(l-1)},\dots,\epsilon^{(K-1)(l-1)})^T\in\mathbb{C}^K,\ \text{for}\ l=1,\dots,K,
\end{align*}
and
\[
P_K=\frac{1}{\sqrt{K}}
\begin{pmatrix}
1 & 1 & \cdots & 1 & 1 \\
1 & \epsilon &\cdots & \epsilon^{K-2} & \epsilon^{K-1} \\
\vdots  & \vdots  & \ddots & \vdots & \vdots\\
1 & \epsilon^{K-2} & \cdots & \epsilon^{(K-2)^2} & \epsilon^{(K-2)(K-1)}\\
1 & \epsilon^{K-1} & \cdots & \epsilon^{(K-1)(K-2)} & \epsilon^{(K-1)^2}
\end{pmatrix}\in\mathbb{M}_K.
\]

For the circulant matrix $B=\textit{circ}\{\mathbf{b}\}$, let
\begin{align}\label{eigenvalues}
\lambda_l=b_1+b_2\epsilon^{l-1}+\dots+b_K\epsilon^{(K-1)(l-1)},\ \text{for}\ l=1,\dots,K.
\end{align}
A simple calculation shows that $BX_l=\lambda_lX_l$. Hence $\lambda_l$ is an eigenvalue of $B$ with normalized eigenvector $X_l$. Since $\{X_1,\dots,X_K\}$ is a linearly independent set of vectors in $\mathbb{C}^K$, all of the eigenvalues of $B$ are given by $\lambda_l$, $l=1,\dots,K$.

\begin{lemma} [cf. \cite{KS-12}]
All circulant matrices have the same ordered set of orthonormal eigenvectors $\{X_l\}$. Moreover, $P_K$ is the diagonalizable matrix.
\end{lemma}

Using these notation, we study the invertibility of $T$.

\begin{lemma}\label{lem-T-0}
There is an $K_0\in\mathbb{N}$ such that for all $K\geq K_0$ and every $\textbf{b}\in\mathbb{R}^{2K}$, there exists a unique vector $\textbf{q}\in\mathbb{R}^{2K}$ and a unique constant $\gamma\in\mathbb{R}$ such that
\begin{equation}
T\textbf{q}=\textbf{b}+\gamma\,\textbf{q}_0,\ \textbf{q}\perp\,\textbf{q}_0.
\end{equation}
Moreover, there is a positive constant $C$ which is independent of $K$ such that
\begin{equation}\label{q-b}
\|\textbf{q}\|_2\leq C\|\textbf{b}\|_2,\ \|\dot{\textbf{q}}\|_2\leq C(\ln K)^{1/2}\|\textbf{b}\|_2,\ \text{and}\ \|\ddot{\textbf{q}}\|_2\leq C(\ln K)^{3/2}\|\textbf{b}\|_2.
\end{equation}
Furthermore, the number of zero (negative, positive) eigenvalues of $T$ is 1 ($K-1$, $K$), respectively.
\end{lemma}
\begin{proof}

Note that    \eqref{q-b} is the Euclidean norm, hence it suffices to perform the analysis of the eigenvalues. To this end, first by \eqref{eigenvalues}, the eigenvalues of $A_1$ are
\begin{align}\label{eigenvalues-A1}
\lambda_{1,l}=-2+\epsilon^{l-1}+\epsilon^{(K-1)(l-1)}=-4\sin^2\frac{(l-1)\pi}{K},\ l=1,\dots,K,
\end{align}
and the eigenvalues of $A_2$ are
\begin{align}\label{eigenvalues-A2}
\lambda_{2,l}=\epsilon^{l-1}-\epsilon^{(K-1)(l-1)}=2i\sin\frac{2(l-1)\pi}{K},\ l=1,\dots,K.
\end{align}

Write $\textit{diag}(c_1,\dots,c_K)$ for a diagonal matrix whose diagonal entries starting in the upper left corner are $c_1,\dots,c_K$. Denote the diagonal matrix of $A_1$ and $A_2$ by
\begin{align*}
D_1=\textit{diag}(\lambda_{1,1},\lambda_{1,2},\dots,\lambda_{1,K})\ \text{and}\
D_2=\textit{diag}(\lambda_{2,1},\lambda_{2,2},\dots,\lambda_{2,K}),\ \text{respectively.}
\end{align*}
Since $P_K$ is the diagonalizable matrix for circulant matrices, we have
\begin{align}\label{matrix-P}
P^{-1}TP=
\begin{pmatrix}
P_K^{-1} & 0\vspace{2mm}\\
0 & P_K^{-1}
\end{pmatrix}
T
\begin{pmatrix}
P_K & 0\vspace{2mm}\\
0 & P_K
\end{pmatrix}
=\begin{pmatrix}
c_1D_1+c_4I & c_2D_2\vspace{2mm}\\
-c_2D_2 & c_3D_1
\end{pmatrix}.
\end{align}
Since the matrix $T$ is real and symmetric, all its eigenvalues are real and satisfy the equations
\begin{align}\label{eq-Lambda}
\Lambda^2-\big[(c_1+c_3)\lambda_{1,l}+c_4\big]\Lambda+\big(c_1\lambda_{1,l}+c_4\big)(c_3\lambda_{1,l})+c_2^2\lambda_{2,l}^2=0,
\end{align}
for $l=1,\dots,K$.
Let
\begin{align*}
\alpha_l=(c_1+c_3)\lambda_{1,l}+c_4,\ \text{and}\
\beta_l=\big(c_1\lambda_{1,l}+c_4\big)(c_3\lambda_{1,l})+c_2^2\lambda_{2,l}^2,\ \forall\,l=1,\dots,K.
\end{align*}
Then by \eqref{eq-c1-c4}, \eqref{eigenvalues-A1} and \eqref{eigenvalues-A2}, we have
\begin{align*}
\alpha_l=(\widehat{d}-1)\frac{K^2}{\pi^2}\sin^2\frac{(l-1)\pi}{K}+(\widehat{d}-m-1)>0,
\end{align*}
and
\begin{align*}
\beta_l=-\Big\{\Big(\widehat{d}\frac{K^2}{\pi^2}-(\widehat{d}-1)^2\Big)\sin^2\frac{(l-1)\pi}{K}
+(m-1)\widehat{d}+1\Big\}\frac{K^2}{\pi^2}\sin^2\frac{(l-1)\pi}{K}\leq0.
\end{align*}
Denote the solutions of \eqref{eq-Lambda} by $\Lambda_{1,l}$ and $\Lambda_{2,l}$ with $\Lambda_{1,l}\leq\Lambda_{2,l}$ for $l=1,\dots,K$. Then
\begin{align*}
\Lambda_{1,l}=\frac{\alpha_l}{2}\Big(-\sqrt{1-\frac{4\beta_l}{\alpha_l^2}}+1\Big)\leq0,\ \Lambda_{2,l}=\frac{\alpha_l}{2}\Big(\sqrt{1-\frac{4\beta_l}{\alpha_l^2}}+1\Big)>0,\ \forall\,l=1,\dots,K.
\end{align*}
In particular, for $l=1$, we have
\begin{align}
\Lambda_{1,1}=0,\
\Lambda_{2,1}=\widehat{d}-m-1.
\end{align}
For $l=2,\dots,K$, by Lemma~\ref{lem-d-R}, we have
\begin{align*}
-\frac{4\beta_l}{\alpha_l^2}
&\leq\frac{\Big(\widehat{d}\frac{K^2}{\pi^2}\sin^2\frac{(l-1)\pi}{K}
+m\widehat{d}\Big)\frac{4K^2}{\pi^2}\sin^2\frac{(l-1)\pi}{K}}
{\Big(\frac{\widehat{d}}{2}\frac{K^2}{\pi^2}\sin^2\frac{(l-1)\pi}{K}+\frac{\widehat{d}}{2}\Big)^2}
\leq\frac{C}{\widehat{d}},
\end{align*}
and
\begin{align*}
-\frac{4\beta_l}{\alpha_l^2}
&\geq\frac{\Big(\frac{\widehat{d}}{2}\frac{K^2}{\pi^2}\sin^2\frac{(l-1)\pi}{K}\Big)\frac{4K^2}{\pi^2}\sin^2\frac{(l-1)\pi}{K}}
{\Big(\widehat{d}\frac{K^2}{\pi^2}\sin^2\frac{(l-1)\pi}{K}+\widehat{d}\Big)^2}
\geq\frac{C}{\widehat{d}},\ \forall\,l=2,\dots,K,
\end{align*}
where $C$ is a positive constant.
Therefore, for all $l=2,\dots,K$,
\begin{align*}
-\Lambda_{1,l}\geq\frac{\widehat{d}}{2}\cdot\frac{C}{\widehat{d}}\geq\frac{1}{2}C,\ \text{and}\
\Lambda_{2,l}\geq\alpha_l\geq\frac{1}{2}\widehat{d},\ \text{for some constant}\ C>0,
\end{align*}
from which we get $\|\textbf{q}\|_2\leq C\|\textbf{b}\|_2$.

Define $\widehat{f}_j=f_{j+1}-f_j$ and $\widehat{g}_j=g_{j+1}-g_j$. Then
\begin{align*}
\left\{
\begin{array}{ll}
c_1(\widehat{f}_j-\widehat{f}_{j-1})+c_2(\widehat{g}_j+\widehat{g}_{j-1})=\phi_j-c_4f_j,
\vspace{1mm}\\
-c_2(\widehat{f}_j+\widehat{f}_{j-1})+c_3(\widehat{g}_j-\widehat{g}_{j-1})=\varphi_j,
\end{array}
\right.
\end{align*}
where $\phi_j=\textbf{b}_j$ and $\varphi_j=\textbf{b}_{K+j}$ for $j=1,\dots,K$.
By a similar argument in the proof of $\|\textbf{q}\|_2\leq C\|\textbf{b}\|_2$, we can get
\[
\|\dot{\textbf{q}}\|_2\leq C(\ln K)^{1/2}\|\textbf{b}\|_2.
\]
When this is done, let $\widetilde{f}_j=f_{j+1}-2f_j+f_{j-1}$ and $\widetilde{g}_j=g_{j+1}-2g_j+g_{j-1}$. Then
\begin{align*}
\left\{
\begin{array}{ll}
c_1\widetilde{f}_j=\phi_j-c_4f_j-c_2(\widehat{g}_j+\widehat{g}_{j-1}),
\vspace{1mm}\\
c_3\widetilde{g}_j=\varphi_j+c_2(\widehat{f}_j+\widehat{f}_{j-1}).
\end{array}
\right.
\end{align*}
Using a similar argument, by the definition of $c_j$'s, we get
\[
\|\ddot{\textbf{q}}\|_2\leq C(\ln K)^{3/2}\|\textbf{b}\|_2.
\]
\end{proof}

Now we are going to prove Lemma \ref{lem-T}. An important observation is that the system $T\textbf{q}=\textbf{b}$ can be seen as the discretization of the following continuous system:
\begin{align}\label{cont}
\left\{
\begin{array}{lll}
-(m+1)f(\theta)+(f''-g')(\theta)+\widehat{d}(f+g')(\theta)=\phi(\theta),\quad \theta\in(0,2\pi),
\vspace{1mm}\\
g(\theta)+(f'-g)(\theta)-\widehat{d}(f'+g'')(\theta)=\varphi(\theta),\quad \theta\in(0,2\pi),
\vspace{1mm}\\
f(0)=f(2\pi),\ f'(0)=f'(2\pi),\ g(0)=g(2\pi),\ g'(0)=g'(2\pi).
\end{array}
\right.
\end{align}

\begin{lemma}\label{lem-8-1}
For $K$ sufficiently large, given $\phi,\varphi$ satisfying $\int_0^{2\pi}\varphi=0$, the system \eqref{cont} has a unique solution $(f,g)$ satisfying $\int_0^{2\pi}g=0$. Moreover, there exists a constant $C>0$ such that
\begin{equation}
\|f\|_{C^2([0,2\pi])}+\|g\|_{C^2([0,2\pi])}\leq C\left(\|\phi\|_{C^0([0,2\pi])}+\|\varphi\|_{C^0([0,2\pi])}\right).
\end{equation}
\end{lemma}

\begin{proof}
Let $h=\widehat{d}(f+g')$, then system \eqref{cont} becomes
\begin{align}\label{eq-f-h}
\left\{
\begin{array}{lll}
f''-mf+\frac{\widehat{d}-1}{\widehat{d}}h=\phi,\\
f'-h'=\varphi.\\
f,h\ \text{are}\ 2\pi\ periodic.
\end{array}
\right.
\end{align}
Since $\int_0^{2\pi}\varphi=0$, from the second equation we get
\begin{equation}\label{h}
h(\theta)=f(\theta)-\int_0^\theta\varphi-c_h.
\end{equation}
Here we take $c_h=\frac{1-\widehat{d}}{2\pi}\int_0^{2\pi}f-\frac{1}{2\pi}\int_0^{2\pi}\int_0^\theta\varphi$
such that $-\int_0^{2\pi}f+\widehat{d}^{-1}\int_0^{2\pi}h=0$. Hence $g$ can be solved by
\begin{equation*}
g(\theta)=-\int_0^\theta f+\widehat{d}^{-1}\int_0^\theta h+c_g,
\end{equation*}
where we take $c_g=\frac{1}{2\pi}\int_0^{2\pi}\int_0^\theta f-\widehat{d}^{-1}\frac{1}{2\pi}\int_0^{2\pi}\int_0^\theta h$ such that $\int_0^{2\pi}g=0$.

By \eqref{h} the first equation in \eqref{eq-f-h} becomes
\begin{equation*}
f''-(m-1+\frac{1}{\widehat{d}})f=\phi+\frac{\widehat{d}-1}{\widehat{d}}\left[\int_0^\theta\varphi-\frac{1}{2\pi}\int_0^{2\pi}\int_0^\theta\varphi\right]-\frac{(\widehat{d}-1)^2}{2\pi \widehat{d}}\int_0^{2\pi}f.
\end{equation*}
To solve the above equation, we first integrate it over $[0,2\pi]$ to get
$
\int_0^{2\pi}f=\frac{1}{\widehat{d}-1-m}\int_0^{2\pi}\phi.
$
Thus
\begin{equation*}
f''-(m-1+\frac{1}{\widehat{d}})f=\left[\phi-\frac{(\widehat{d}-1)^2}{\widehat{d}(\widehat{d}-1-m)}\frac{1}{2\pi}\int_0^{2\pi}\phi\right]+\frac{\widehat{d}-1}{\widehat{d}}\left[\int_0^\theta\varphi-\frac{1}{2\pi}\int_0^{2\pi}\int_0^\theta\varphi\right].
\end{equation*}
Note that $m-1+\frac{1}{\widehat{d}}>0$, by the boundary condition, $f$ is uniquely given and satisfies $\|f\|_{C^2([0,2\pi])}\leq C\left(\|\phi\|_{C^0([0,2\pi])}+\|\varphi\|_{C^0([0,2\pi])}\right)$. Therefore,
\begin{equation*}
g(\theta)=-\int_0^\theta f+\widehat{d}^{-1}\int_0^\theta h+c_g,
\end{equation*}
satisfies the same inequality, where\begin{equation*}
h(\theta)=f(\theta)+\frac{(\widehat{d}-1)}{2\pi(\widehat{d}-1-m)}\int_0^{2\pi}\phi-\left[\int_0^\theta\varphi-\frac{1}{2\pi}\int_0^{2\pi}\int_0^\theta\varphi\right].
\end{equation*}
\end{proof}

\begin{remark}
Let $c=\sqrt{m-1+\frac{1}{\widehat{d}}}$, by the equation of $f$, we can get
\[
f(\theta)=\int_0^{2\pi}\Big[\frac{(\widehat{d}-1)^2}{\widehat{d}(\widehat{d}-m-1)}\frac{1}{2\pi c^2}-G_0(\theta,s)\Big]\phi(s)\,ds
+\int_0^{2\pi}\frac{\widehat{d}-1}{\widehat{d}}\Big[\int_0^tG_0(\theta,t)-\frac{s}{2\pi c^2}\Big]\varphi(s)\,ds,
\]
where
\begin{align*}
G_0(\theta,s)=
\left\{
\begin{array}{ll}
\frac{1}{2c(e^{2\pi c}-1)}\big[e^{2\pi c}e^{c(\theta-s)}+e^{-c(\theta-s)}\big],\ &\text{if}\ \theta\leq s,\\
\frac{1}{2c(e^{2\pi c}-1)}\big[e^{c(\theta-s)}+e^{2\pi c}e^{-c(\theta-s)}\big],\ &\text{if}\ \theta> s.\\
\end{array}
\right.
\end{align*}
Actually there is a Green's matrix
\begin{equation*}
G(\theta,t)=
\begin{pmatrix}
G_{11}(\theta,s) & G_{12}(\theta,s)
\vspace{2mm}\\
G_{21}(\theta,s) & G_{22}(\theta,s)
\end{pmatrix}
\end{equation*}
such that
\begin{equation*}
f(\theta)=\int_0^{2\pi}G_{11}(\theta,s)\phi(s)\,ds+\int_0^{2\pi}G_{12}(\theta,s)\varphi(s)\,ds,
\end{equation*}
and
\begin{equation*}
g(\theta)=\int_0^{2\pi}G_{21}(\theta,s)\phi(s)\,ds+\int_0^{2\pi}G_{22}(\theta,s)\varphi(s)\,ds.
\end{equation*}
\end{remark}

\begin{lemma}\label{lem-T-1}
Under the assumption of Lemma~\ref{lem-T-0}, there is a positive constant $C$ which is independent of $K$ such that
\begin{equation}\label{q-b}
\|\textbf{q}\|_*\leq C(\ln K)^2\|\textbf{b}\|_\infty.
\end{equation}
\end{lemma}

\begin{proof}
\noindent{\bf{Claim 1:}} There is a positive constant $C$ (independent of $K$) such that
\begin{equation}\label{q-b-1}
\|\textbf{q}\|_\infty\leq C\|\textbf{b}\|_\infty.
\end{equation}
To prove it, we only need to consider the case $\textbf{b}\perp\textbf{q}_0$. For $j=1,\dots,K$, we define
\begin{align*}
\textbf{q}_j=(f_{1,j},\dots,f_{K,j},g_{1,j},\dots,g_{K,j})^T,
\end{align*}
where
\begin{equation*}
f_{l,j}=\frac{2\pi}{K}G_{11}(\theta_l,\theta_j),\
g_{l,j}=\frac{2\pi}{K}G_{21}(\theta_l,\theta_j).
\end{equation*}
This corresponds to take $\phi=\frac{2\pi}{K}\delta(\theta-\theta_j),\varphi=0$ for $j=1,\dots,K$, where $\delta$ is the delta function in the distribution theory.
For $j=K+2,\dots,2K$, we define
\begin{align*}
\textbf{q}_j=(f_{1,j},\dots,f_{K,j},g_{1,j},\dots,g_{K,j})^T,
\end{align*}
where
\begin{equation*}
f_{l,j}=-\frac{2\pi}{K}G_{12}(\theta_l,\theta_1)+\frac{2\pi}{K}G_{12}(\theta_l,\theta_j),\
g_{l,j}=-\frac{2\pi}{K}G_{22}(\theta_l,\theta_1)+\frac{2\pi}{K}G_{22}(\theta_l,\theta_j).
\end{equation*}
This corresponds to take $\phi=0,\varphi=-\frac{2\pi}{K}\delta(\theta-\theta_1)+\frac{2\pi}{K}\delta(\theta-\theta_j)$ for $j=K+2,\dots,2K$.
By the property of Green's matrix, for $j=1,\dots,K$, we have
\begin{equation*}
T\textbf{q}_j=\vec{e}_j+\vec{\tau}_j;
\end{equation*}
for $j=K+2,\dots,2K$, we have
\begin{equation*}
T\textbf{q}_j=-\vec{e}_{K+1}+\vec{e}_j+\vec{\tau}_j,
\end{equation*}
where $\vec{e}_j=(\delta_{j,1},\dots,\delta_{j,2K})^T$ is standard orthonormal basis of $\mathbb{R}^{2K}$ and $\vec{\tau}_j=O(K^{-2})$ is the local truncation error for the Green's matrix in the finite difference method. Since $\{\vec{e}_1,\dots,\vec{e}_K,-\vec{e}_{K+1}+\vec{e}_{K+2},-\vec{e}_{K+1}+\vec{e}_{2K}\}$ is a basis of $\{\textbf{b}\in\mathbb{R}^{2K}\,|\,\textbf{b}\perp\textbf{q}_0\}$ and write $\textbf{b}=\sum_{j=1}^Kb_j\vec{e}_j+\sum_{j=K+2}^{2K}b_j(-\vec{e}_{K+1}+\vec{e}_j)$. Then we get $\textbf{q}=\widehat{\textbf{q}}+\widetilde{\textbf{q}}$, where $\widehat{\textbf{q}}=\sum_{j=1}^Kb_j\textbf{q}_j+\sum_{j=K+2}^{2K}b_j\textbf{q}_j$ and $T\widetilde{\textbf{q}}=\sum_{j\neq K+1}b_j\vec{\tau}_j$.

On one hand, by the property of Green's matrix, we get $\|\widehat{\textbf{q}}\|_\infty\leq C\|\textbf{b}\|_\infty$. On the other hand, by Lemma~\ref{lem-T-0}, we get
\[
\|\widetilde{\textbf{q}}\|_\infty
\leq\|\widetilde{\textbf{q}}\|_2
\leq C\|\sum_{j\neq K+1}b_j\vec{\tau}_j\|_2\leq CK^{1/2}\|\sum_{j\neq K+1}b_j\vec{\tau}_j\|_\infty
\leq CK^{-3/2}\|\textbf{b}\|_\infty.
\]
Combining these two estimates, we get $\|\textbf{q}\|_\infty\leq C\|\textbf{b}\|_\infty$.

\noindent{\bf{Claim 2:}} There is a positive constant $C$ (independent of $K$) such that
\begin{align}\label{estimate-q-b}
\|\textbf{q}\|_*=\|\textbf{q}\|_\infty+\|\dot{\textbf{q}}\|_\infty+\|\ddot{\textbf{q}}\|_\infty\leq C(\ln K)^2\|\textbf{b}\|_\infty.
\end{align}

\noindent{\bf{Proof of Claim 2:}}

Define $\widehat{f}_j=f_{j+1}-f_j$ and $\widehat{g}_j=g_{j+1}-g_j$. Then
\begin{align*}
\left\{
\begin{array}{ll}
c_1(\widehat{f}_j-\widehat{f}_{j-1})+c_2(\widehat{g}_j+\widehat{g}_{j-1})=\phi_j-c_4f_j,
\vspace{1mm}\\
-c_2(\widehat{f}_j+\widehat{f}_{j-1})+c_3(\widehat{g}_j-\widehat{g}_{j-1})=\varphi_j.
\end{array}
\right.
\end{align*}
By using a similar argument, we can get
\begin{equation}
\|\dot{\textbf{q}}\|_\infty\leq C\|\textbf{b}\|_\infty+(\ln K)\|\textbf{q}\|_\infty
\leq C(\ln K)\|\textbf{b}\|_\infty.
\end{equation}
Let $\widetilde{f}_j=f_{j+1}-2f_j+f_{j-1}$ and $\widetilde{g}_j=g_{j+1}-2g_j+g_{j-1}$. Then
\begin{align*}
\left\{
\begin{array}{ll}
c_1\widetilde{f}_j=\phi_j-c_4f_j-c_2(\widehat{g}_j+\widehat{g}_{j-1}),
\vspace{1mm}\\
c_3\widetilde{g}_j=\varphi_j+c_2(\widehat{f}_j+\widehat{f}_{j-1}).
\end{array}
\right.
\end{align*}
Similarly we obtain
\begin{equation}
\|\ddot{\textbf{q}}\|_\infty\leq C\|\textbf{b}\|_\infty+(\ln K)\|\textbf{q}\|_\infty+(\ln K)\|\dot{\textbf{q}}\|_\infty
\leq C(\ln K)^2\|\textbf{b}\|_\infty.
\end{equation}

Actually if one can show that $\|\widehat{\textbf{q}}\|_*\leq C\|\textbf{b}\|_\infty$, then we can get $\|\textbf{q}\|_*\leq C\|\textbf{b}\|_\infty$ by Lemma~\ref{lem-T-0} since the local truncation error for the Green's matrix is $O(K^{-2})$.

\end{proof}

Now we can use Lemma~\ref{lem-T-1} to prove Lemma~\ref{lem-T}.

\begin{proof}[Proof of Lemma \ref{lem-T}]

To prove Lemma \ref{lem-T}, it suffices to prove the a priori estimate~\eqref{q-b-0}. let $\gamma=-(\textbf{b}\cdot\textbf{q}_0)/(\textbf{q}_1\cdot\textbf{q}_0)$. By Lemma~\ref{lem-M}, for $\textbf{q}$ satisfies \eqref{cond-q}, we have
\[
R^{-1}\textbf{q}_1=c_0\textbf{q}_0+O(KR^{-m}),
\]
which implies that $\|R^{-1}\textbf{q}_1\|_\infty\leq C$ and $|R^{-1}\textbf{q}_1\cdot\textbf{q}_0|\geq CK$. Hence $\|\gamma\textbf{q}_1\|_\infty\leq C\|\textbf{b}\|_\infty$. Therefore, by Lemma~\ref{lem-T-1}, we have
\[
\|\vec{\textbf{q}}\|_*\leq C(\ln K)^2\|\textbf{b}+\gamma\textbf{q}_1\|_\infty\leq C(\ln K)^2\|\textbf{b}\|_\infty.
\]

\end{proof}


\section{Appendix B: Energy expansion}

In this section, we give the energy expansion of $\mathcal{E}(U+\phi)$ and prove Lemma \ref{lem-energy}.

\begin{proof}[Proof of Lemma \ref{lem-energy}]
By \eqref{energy}, we get
\begin{align*}
\mathcal{E}(u)&=\underbrace{\sum_{j=1}^K \mathcal{E}(w_{Q_j})}_{J_1}+\underbrace{\frac{1}{2}\sum_{i\neq j}\int_{\mathbb{R}^N}\left(\nabla w_{Q_i}\nabla w_{Q_j}+V(x)w_{Q_i}w_{Q_j}\right)\,dx}_{J_2}\\
&\quad+\underbrace{\frac{1}{p+1}\int_{\mathbb{R}^N}\Big\{\sum_{j=1}^Kw_{Q_j}^{p+1}-\big(\sum_{j=1}^Kw_{Q_j}\big)^{p+1}\Big\}\,dx}_{J_3}.
\end{align*}

{{\bf Claim 1:}} By \eqref{cond-wy}, there are positive constants $I_0$ and $a_0$ such that
\begin{align}
J_1=KI_0+a_0\big(1+o(1)\big)\sum_{j=1}^K|Q_j|^{-m}.
\end{align}
Indeed, by the definition of the energy functional, \eqref{cond-wy} and Taylor's expansion,
\begin{align*}
J_1
&=KI_0+\frac{1}{2}\sum_{j=1}^K\int_{\mathbb{R}^N}\left(V(x)-1\right)w_{Q_j}^2\,dx
=KI_0+a_0\sum_{j=1}^K\Big(|Q_j|^{-m}+O(R^{-m-\sigma})\Big),
\end{align*}
where
\begin{equation*}
I_0=\big(\frac{1}{2}-\frac{1}{p+1}\big)\int_{\mathbb{R}^N}w^{p+1}\,dx,\ \text{and}\
a_0=\frac{a}{2}\int_{\mathbb{R}^N}w^2\,dx.
\end{equation*}

%
%

{{\bf Claim 2:}} By \eqref{cond-wy} we have
\begin{equation}
J_2=\frac{1}{2}\sum_{i\neq j}\int_{\mathbb{R}^N}w_{Q_j}^pw_{Q_i}\,dx+O(KR^{-2m}).
\end{equation}
Indeed, the term $J_2$ can be divided into two parts:
\begin{equation*}
J_2=\underbrace{\frac{1}{2}\sum_{i\neq j}\int_{\mathbb{R}^N}w_{Q_j}^pw_{Q_i}\,dx}_{J_{21}}+\underbrace{\frac{1}{2}\sum_{i\neq j}\int_{\mathbb{R}^N}\left(V(x)-1\right)w_{Q_j}w_{Q_i}\,dx}_{J_{22}}.
\end{equation*}
%
For $J_{22}$, by \eqref{cond-wy}, we have
\begin{align*}
|J_{22}|&
\leq CR^{-m}\sum_{i\neq j}e^{-|Q_j-Q_i|}|Q_j-Q_i|^{-(N-3)/2}
\leq CKR^{-m}e^{-d}d^{-(N-3)/2}
\leq CKR^{-2m}.
\end{align*}

{{\bf Claim 3:}} Let $\textbf{Q}\in\Lambda_K$, we have
\begin{equation}
J_3=-\sum_{i\neq j}\int_{\mathbb{R}^N}w_{Q_j}^pw_{Q_i}\,dx+O(Ke^{-\min\{2,\frac{p+1}{2}\}d}d^{-\frac{N-3}{2}}).
\end{equation}
Indeed, write
\begin{equation*}
J_3=\frac{1}{p+1}\int_{\mathbb{R}^N}\Big\{\sum_{j=1}^Kw_{Q_j}^{p+1}-\big(\sum_{j=1}^Kw_{Q_j}\big)^{p+1}\Big\}\,dx=-\int_{\mathbb{R}^N}E_3\,dx.
\end{equation*}
For $x\in\Omega_{K+1}^\ell$, where $\ell\in\mathbb{N}$ chosen later, we have
\begin{align*}
|E_3(x)|&\leq\frac{1}{p+1}\Big\{\sum_{j=1}^Kw_{Q_j}^{p+1}+\big(\sum_{j=1}^Kw_{Q_j}\big)^{p+1}\Big\}
\leq\frac{1}{p+1}\Big(\sum_{j=1}^Kw_{Q_j}^{p+1}+K^p\sum_{j=1}^Kw_{Q_j}^{p+1}\Big).
\end{align*}
By choosing $\ell\geq\frac{4(p+2m+3)}{pm}$ (but independent of $K$), we get
\begin{align*}
\int_{\Omega_{K+1}^\ell}|E_3|\,dx
&\leq CK^p\sum_{j=1}^K\int_{\Omega_{K+1}^\ell}w_{Q_j}^{p+1}\,dx
\leq CK^{p+1}w^{p}(\ell\rho/2)
\leq CK^{-2m-2}.
\end{align*}
For $x\in\Omega_j^\ell$, $j=1,2,\dots,K$, by a similar argument in the proof of Lemma~\ref{lem-error}, we have
\begin{align*}
\Big|E_3-w_{Q_j}^p\sum_{i\neq j}w_{Q_i}\Big|
&\leq C\ell^{(N-1)(p-1)}w_{Q_j}^{p-1}\Big(\sum_{i\neq j}w_{Q_i}\Big)^2.
\end{align*}
Applying Lemma \ref{lem-act}, we get
\begin{align*}
\Big|J_3+\sum_{i\neq j}\int_{\mathbb{R}^N}w_{Q_j}^pw_{Q_i}\Big|
&\leq C\sum_{j=1}^K\int_{\Omega_j^\ell}w_{Q_j}^{p-1}\Big(\sum_{i\neq j}w_{Q_i}\Big)^2\,dx+CK^{-2m-2}\\
&\leq CKe^{-\min\{2,\frac{p+1}{2}\}d}d^{-\frac{N-3}{2}}
\leq CK^{-\min\{2,\frac{p+1}{2}\}m+1}(\ln K)^{1/2},
\end{align*}
which implies Claim 3.

Combining our Claim 1, Claim 2 and Claim 3, the desired result follows from Lemma \ref{lem-act}.
\end{proof}

At the last, by \eqref{energy}, Lemma \ref{lem-error} and Proposition \ref{prop-3-1},
\begin{align*}
\mathcal{E}(U+\phi)&=\mathcal{E}(U)
+\frac{1}{2}\int_{\mathbb{R}^N}\Big\{|\nabla\phi|^2+V(x)\phi^2\Big\}\,dx
+\int_{\mathbb{R}^N}\big(\nabla U\nabla\phi+V(x)U\phi\big)\,dx\\
&\quad-\frac{1}{p+1}\int_{\mathbb{R}^N}\Big\{(U+\phi)_+^{p+1}-U^{p+1}\Big\}\,dx\\
&=\mathcal{E}(U)
+\frac{1}{2}\int_{\mathbb{R}^N}\Big\{(U+\phi)_+^p-U^p+E\Big\}\phi\,dx\\
&\quad-\frac{1}{p+1}\int_{\mathbb{R}^N}\Big\{(U+\phi)_+^{p+1}-U^{p+1}-(p+1)U^p\phi\Big\}\,dx\\
&=\mathcal{E}(U)+O(K\|\phi\|_{**}^2)+O(K\|\phi\|_{**}\|E\|_{**})\\
&=KI_0+(a_0+o(1))\sum_{j=1}^K|Q_j|^{-m}
-\frac{1}{2}\sum_{i\neq j}(\gamma_0+o(1))w(|Q_i-Q_j|)\\
&\quad+O\big(K^{-\min\{2,\frac{p+1}{2}\}m+1}(\ln K)^{1/2}\big).
\end{align*}


\end{document}